\documentclass[12pt]{article}
\usepackage[utf8]{inputenc}
\usepackage{amsfonts}
\usepackage{amssymb}
\usepackage{amsmath}
\usepackage{stmaryrd}  
\usepackage{wasysym}
\usepackage{multirow} 
\usepackage{footmisc}
\usepackage{enumitem}
\usepackage{graphicx}
\usepackage{subcaption}
\usepackage{wrapfig}
\usepackage{csquotes}
\usepackage{xcolor} 
\usepackage[usestackEOL]{stackengine}

\newtheorem{theorem}{Theorem}[section]
\usepackage[colors]{optsys}
\multiopssep{setco}{}{\{}{:}{\}}                              
\newcommand{\spacedim}{d}

\newcommand{\atom}{\primal} 
\newcommand{\atomopt}{\atom} 
\newcommand{\Atom}{\Primal} 
\newcommand{\AtomicSet}{\Atom_{k}}

\renewcommand{\opt}{^\sharp}
\newcommand{\OptimalSupports}[2]{{\mathcal K}\opt_{#1}\np{#2}}

\renewcommand{\PRIMAL}{\RR^{\spacedim}} 
\renewcommand{\DUAL}{\RR^{\spacedim}} 
\renewcommand{\primalbis}{z}

\renewcommand{\Vset}{\ic{1,{\spacedim}}}
\newcommand{\Vsetzero}{\ic{0,{\spacedim}}}

\renewcommand{\HILBERT}{{\mathcal H}}

\newcommand{\BallDual}{\Ball_{\star}}

\renewcommand{\TripleNorm}{\DoubleNorm}
\renewcommand{\TripleNormDual}{\DoubleNormDual}
\renewcommand{\TripleNormBall}{\Ball} 
\renewcommand{\TripleNormDualBall}{\BallDual}
\renewcommand{\TripleNormSphere}{\Sphere}

\newcommand{\TopSupscript}{\top}
\newcommand{\SupportSupscript}{\TopSupscript\!\star} 
\renewcommand{\TopNorm}[2]{{#1}_{(#2)}^{\TopSupscript}}
\renewcommand{\TopDualNorm}[2]{{#1}_{\star,(#2)}^{\TopSupscript}}

\renewcommand{\SupportNorm}[2]{{#1}_{(#2)}^{\SupportSupscript}}
\renewcommand{\SupportDualNorm}[2]{{#1}_{\star,(#2)}^{\SupportSupscript}}

\renewcommand{\LpTopNorm}[3]{\Norm{#1}_{#2,#3}^{\TopSupscript}}
\renewcommand{\LpSupportNorm}[3]{\Norm{#1}_{#2,#3}^{\SupportSupscript}}
\renewcommand{\LpTopBall}[3]{{#1}_{#2,#3}^{\TopSupscript}}
\renewcommand{\LpSupportBall}[3]{{#1}_{#2,#3}^{\SupportSupscript}}

\newcommand{\ProjectionIndexSubset}{\projection_{\IndexSubset}}
\newcommand{\ProjectionIndexSubsetOpt}{\projection_{\IndexSubset\opt}}

\renewcommand{\cardinal}[1]{{\footnotesize\textrm{card}}#1}
\newcommand{\intercardK}{\stackunder{$\bigcap$}{\tiny $\cardinal{K} \leq k$}}
\newcommand{\unioncardK}{\stackunder{$\bigcup$}{\tiny $\cardinal{K} \leq k$}}
\newcommand{\supcardK}{\stackunder{$\mathrm{sup}$}{\tiny $\cardinal{K} \leq k$}}

\usepackage{fullpage}
\usepackage{authblk}
\graphicspath{{./}}

\title{Geometry of Sparsity-Inducing Norms}

\author[1]{Jean-Philippe Chancelier}
\author[1]{Michel De~Lara}
\author[2]{Antoine Deza}
\author[3]{Lionel~Pournin}

\affil[1]{CERMICS, ENPC, Institut Polytechnique de Paris, CNRS, Marne-la-Vallée, France}
\affil[2]{McMaster University, Hamilton, Ontario, Canada}
\affil[3]{Universit{\'e} Paris 13, Villetaneuse, France}

\begin{document}

\maketitle

\begin{abstract}
  Sparse optimization seeks an optimal solution with few nonzero entries.  To
  achieve this, it is common to add to the criterion a penalty term proportional
  to the $\ell_1$-norm, which is recognized as the archetype of sparsity-inducing
  norms. In this approach, the number of nonzero entries is not controlled a
  priori.  By contrast, in this paper, our motivation is to find an optimal solution
  with at most~$k$ nonzero coordinates (or for short, $k$-sparse vectors), where
  $k$ is a given sparsity threshold (or ``sparsity budget''). For this purpose, we
  study the class of generalized $k$-support dual~norms that arise from any given
  so-called source norm.  When added as a penalty term, we provide conditions under which
  such generalized $k$-support dual~norms promote $k$-sparse solutions.  The result
  follows from an analysis of the exposed faces of closed convex sets generated
  by $k$-sparse vectors, and of how primal support identification can be deduced
  from dual information.  Finally, we study some of the geometric properties of
  the unit balls for the $k$-support dual~norms and their dual norms when the source
  norm belongs to the family of $\ell_p$-norms. In particular, we show a
  striking structural property: every proper face of the unit balls for the
  $k$-support dual~norms is a hypersimplex, i.e., the convex hull of $0/1$-valued
  points with the same $\lzero$-norm. 
\end{abstract}


{{\bf 2020 Mathematics Subject Classification (MSC2020):}
    49N15,  	
    90C25,  	
    52A05,  	
    52A21.  	
  }

\section{Introduction}
\label{CDLDP.sec.0}

After motivating the paper in~\S\ref{Motivation},
we discuss how it relates to the literature in~\S\ref{Related_literature},
and we finally outline the organization of the paper
in~\S\ref{Organization_of_the_paper}. 

\subsection{Motivation}
\label{Motivation}

In 1996, Tibshirani~\cite{Tibshirani:1996} proposed least-squares regression with
an $\ell_1$-norm penalty to achieve sparsity in least-squares minimization.
Figure~\ref{fig:Figure2_Tibshirani} is the replica of
\cite[Figure~2]{Tibshirani:1996}, which provides insight regarding why
corresponding optimal solutions are sparse (we copy the comments
of~\cite[Figure~2]{Tibshirani:1996} with additional precisions in
brackets~[$\cdots$]):

\blockquote{``The elliptical contours of this function [quadratic
criterion] are shown by the full curves in Fig.~2(a); they are centred at the
OLS [optimal least-square] estimates; the constraint region [$\ell_1$-ball in
dimension~2] is the rotated square. The lasso solution is the first place that
the contours touch the square, and this will sometimes occur at a corner,
corresponding to a zero coefficient. The picture for ridge regression is shown
in Fig. 2(b): there are no corners for the contours to hit and hence zero
solutions will rarely result.''}
%
\begin{figure}[ht!]
  \begin{center}
    \includegraphics[width=0.50\textwidth]{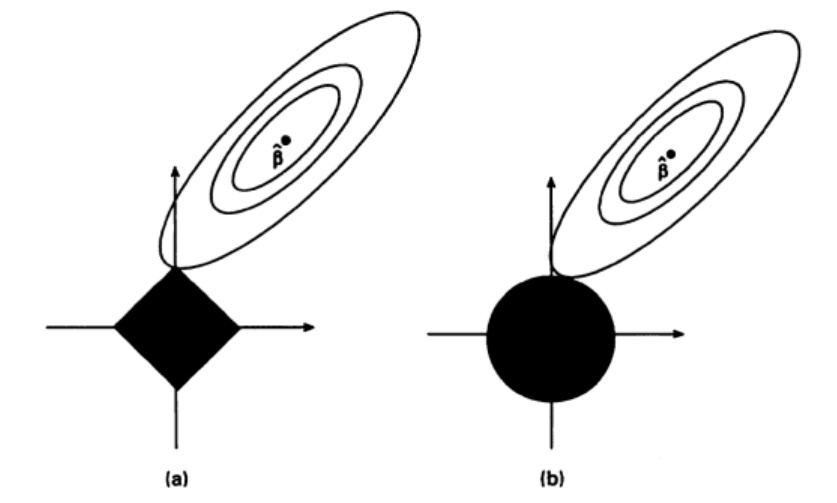}
    \caption{Replica of~\cite[Figure~2]{Tibshirani:1996}
    \label{fig:Figure2_Tibshirani}}
  \end{center}
\end{figure}
Thus, as the kinks of the $\ell_1$-ball are located at sparse points, it is common
to say that the $\ell_1$-norm is {\em sparsity-inducing}.
Figure~\ref{Two_examples_of_unit_balls_with_kinks_located_at_sparse_points}
shows two examples of unit balls with kinks located at sparse points. Both of
them arise from norms that are studied in the sequel.

\begin{figure}[ht!]
        \begin{center}
          \includegraphics[scale=1.3]{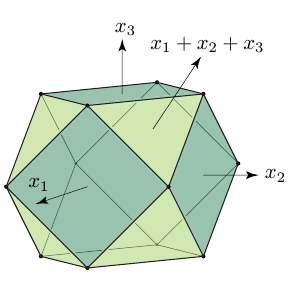}\hspace{5\bigskipamount}
          \includegraphics[scale=1.3]{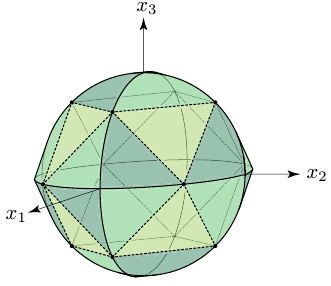}
          \caption{Two examples of unit balls with kinks located at sparse points
            \label{Two_examples_of_unit_balls_with_kinks_located_at_sparse_points}}
        \end{center}
\end{figure}

A natural question that arises from the comments
of~\cite[Figure~2]{Tibshirani:1996} is: What could be mathematical conditions
for inducing sparsity?

Going beyond least-squares regression with an $\ell_1$-norm penalty, one can
consider optimization problems of the form
\( \min_{\primal\in\RR^{\spacedim}} \bp{\fonctionprimal\np{\primal} + \gamma
  \TripleNorm{\primal}} \), where
\( \fonctionprimal : \RR^{\spacedim} \to \RR \) is a smooth convex function,
$\gamma >0$ and $\TripleNorm{\cdot}$ is a norm with unit ball~$\Ball$.  This is the
approach taken in the
papers~\cite{Bach-Jenatton-Mairal-Obozinski:2012,Jenatton-Audibert-Bach:2011},
where the terminology ``sparsity-inducing norm'' has been introduced.  As all
functions take finite values --- and as $\TripleNorm{\cdot}$ is the support
function~$\SupportFunction{\PolarSet{\Ball}}$ of the polar
set~$\PolarSet{\Ball}$ --- a solution\footnote{%
  The superscript~$\opt$ will be used in all the paper to denote an optimal
  solution, that is, a solution to an optimization problem (here, minimization).
  \label{ft:optimal_solution}
}
~$\primal\opt$ of the above problem is
characterized by the Fermat rule
\begin{subequations}
  \begin{equation}
    \label{eq:Fermat_rule}
  0 \in \nabla\fonctionprimal\np{\primal\opt}
  + \gamma \partial\SupportFunction{\PolarSet{\Ball}}\np{\primal\opt}
  \eqfinv  
\end{equation}
where \( \partial\SupportFunction{\PolarSet{\Ball}}\np{\primal\opt} \) is the face
\( \ExposedFace(\PolarSet{\Ball},\primal\opt) \) of the unit ball~$\PolarSet{\Ball}$ exposed
by~\( \primal\opt \).
Thus, the optimality condition reads
\begin{equation}
      \label{eq:Fermat_rule_with_polar_ball}
  -\nabla\fonctionprimal\np{\primal\opt}
  \in \gamma\ExposedFace(\PolarSet{\Ball},\primal\opt)
\end{equation}
and, by polarity~\cite[Theorem~5.1]{Fan-Jeong-Sun-Friedlander:2020}, this is
equivalent to \(\gamma=\TripleNormDual{\nabla\fonctionprimal\np{\primal\opt}}\)
(dual norm) and\footnote{%
If we were to develop an algorithm to find the optimal
  solution~$\primal\opt$,
  this would require knowing the exposed faces of the unit
  ball~$\PolarSet{\Ball}$,
  as in~\eqref{eq:Fermat_rule} and \eqref{eq:Fermat_rule_with_polar_ball}.
However, this paper only studies the exposed faces of the unit
  ball~$\Ball$, as in~\eqref{eq:Fermat_rule_with_original_ball}.}  
\begin{equation}
       \label{eq:Fermat_rule_with_original_ball}
 \frac{\primal\opt}{\TripleNorm{\primal\opt}} \in \ExposedFace(\Ball, -\nabla\fonctionprimal\np{\primal\opt})
  \eqfinp
\end{equation}
\end{subequations}
One could say that the considered norm is {\em sparsity-inducing} if information
about the support of~$\primal\opt$ could be obtained from information about the
exposed faces \( \ExposedFace(\Ball, -\nabla\fonctionprimal\np{\primal\opt}) \) of
the unit ball~$\Ball$ for that norm. This is the approach that we consider in
this paper. More precisely, we analyze the exposed faces of some special convex
sets, and in particular of the unit balls of certain norms, and relate them to
sparsity.

\subsection{Related literature}
\label{Related_literature}

Our work is related to different trends in the literature.
As said above, the terminology ``sparsity-inducing norm'' has been introduced in
the
papers~\cite{Bach-Jenatton-Mairal-Obozinski:2012,Jenatton-Audibert-Bach:2011},
which focus on algorithmic issues, whereas we focus on geometric aspects.  In
particular, we study how the gradient --- at a solution, of the original smooth
function to be minimized --- provides relevant (dual) information about the
sparsity of the (primal) solution.
\medskip

Sparsity is also examined for the solutions of undetermined linear systems,
and we emphasize the three papers
\cite{Candes:2008,Chandrasekaran-Recht-Parrilo-Willsky:2012,Boyer-Chambolle-DeCastro-Duval-deGournay-Weiss:2019}.

The paper~\cite{Candes:2008} studies the solutions of an undetermined linear
system in the context of compressed sensing; it provides a sufficient property
of the sensing matrix, the restricted isometry property, which ensures that the
minimal $\ell_1$-norm solution coincides with the sparse solution.

In~\cite{Chandrasekaran-Recht-Parrilo-Willsky:2012}, it is explained how to
design so-called ``atomic norms'' which promote sparsity, but with respect to a
given (compact) atomic set: the convex closure of the set of atoms is the unit
ball of the atomic norm, and the authors stress that it is ``the favorable facial structure of the atomic norm ball
that makes the atomic norm a suitable convex heuristic to recover simple
models''.  The norms that we present in
Sect.~\ref{The_case_of_generalized_top-k_and_k-support_norms} are atomic norms
where the atomic set is especially designed to provide solutions with an a
priori given ``sparsity budget'' (number of nonzero entries bounded above by a
given integer). 
More precisely, the atomic set that we consider is as follows:
we project the unit ball of a given so-called source norm onto the
different subspaces of $k$-sparse vectors.
Our main contribution is to focus on the geometric description of the facial structure of
the unit balls of these special atomic norms, in relation to the facial
structure of the unit ball of the source norm.
As said above, we study in
particular how a gradient provides relevant (dual) information about sparsity of
the (primal) solution. By contrast,
\cite{Chandrasekaran-Recht-Parrilo-Willsky:2012} focuses more on measuring the
Gaussian width of the tangent cones as a way to achieve more or less
sparsity.

The paper \cite{Boyer-Chambolle-DeCastro-Duval-deGournay-Weiss:2019} also
stresses the role of faces in identifying solutions of undetermined linear
systems that can be expressed as convex combinations of a small number of atoms.
Where \cite{Boyer-Chambolle-DeCastro-Duval-deGournay-Weiss:2019} studies faces,
we focus on \emph{exposed} faces and on how dual information is related to
sparsity of a primal solution.
\medskip

Let us mention three more works, and how they each connect to the present paper.

The question of decomposing a vector as a convex conical combination of
elementary atoms has been studied in~\cite{Fan-Jeong-Sun-Friedlander:2020}, with
a special role given to the so-called alignment, that is, to normal cones and
exposed faces. Our approach intersects that of
\cite{Fan-Jeong-Sun-Friedlander:2020}, but with a focus on the classic sparsity
along coordinate axis and with the goal to describe the geometry of some unit
balls \cite{Chancelier-DeLara:2022_OSM_JCA}.

In~\cite{Fadili-Malick-Peyre:2018}, the focus is put on stratifying the primal
space. Then, an optimal primal-dual pair
\( \np{\primal\opt,\nabla\fonctionprimal\np{\primal\opt}} \) can provide information on the
strata to which~$\primal\opt$ belongs. General regularizers (and not only norms)
are studied.
To connect with our work, we should examine the notion mirror-stratification in
the case of norms, and check whether the exposed faces and normal cones lattices
are in a mirror-stratificable conjugacy. This analysis goes beyond the current
paper. 

%
Convex structured sparsity with norms is the object
of~\cite{Obozinski-Bach:hal-01412385}.
From a combinatorial function (that maps support of vectors onto extended real
numbers) and a \( \ell_p \)-norm, the authors present a general way to produce
a new norm \cite[Equation~(2)]{Obozinski-Bach:hal-01412385}. Then, they study properties of this class of norms, relate them to
well-known norms (by varying the combinatorial function), display visual
representation of the exposed faces of some unit balls \cite[Figures~5 and~6]{Obozinski-Bach:hal-01412385}, propose
variational representations \cite[Section~5]{Obozinski-Bach:hal-01412385} and algorithms to compute proximal
operators \cite[\S~6.3]{Obozinski-Bach:hal-01412385}.
They also provide a probabilistic ``support recovery'' result \cite[Proposition~7]{Obozinski-Bach:hal-01412385}.
In our paper, from a given ``sparsity budget'' (an integer less than the
ambient dimension) and \emph{any} norm, we present a general way to produce
a new norm (\S\ref{Face_characterization_and_support_identification}).
So, compared to~\cite{Obozinski-Bach:hal-01412385}, we are less
general in that our sparsity budget~$k$ is a special case of combinatorial function
--- one that attributes the value~1 to any support with cardinality less
than~$k$, 0 to the empty set, and $+\infty$ else (see
\cite[Illustration~2]{Obozinski-Bach:hal-01412385} for a related expression) ---
but we are more general in that we start with \emph{any} norm, and do not
restrict to \( \ell_p \)-norms (although we provide a more thorough analysis in
Sect.~\ref{Geometry_of_lptopnorm{q}{k}_and__lpsupportnorm{p}{k}s_unit_balls} in
the case of \( \ell_p \)-norms).
We study exposed faces and how they relate to sparsity
(Theorem~\ref{th:Support_identification},
Proposition~\ref{pr:Support_identification_generalized_k-support_norm})
with a special emphasis the case of orthant-monotonic norms
(Proposition~\ref{pr:Support_identification_generalized_k-support_norm_orthant-monotonic})
and orthant-strictly monotonic norms
(Proposition~\ref{pr:Support_identification_generalized_k-support_norm_orthant-strictly_monotonic}).
This is the first main difference with~\cite{Obozinski-Bach:hal-01412385}:
we focus on the geometric analysis of exposed faces of new balls obtained from
any norm, in relation to a given sparsity budget, whereas
\cite{Obozinski-Bach:hal-01412385} is less focused on the geometry of faces
\cite[Figures~5 and~6]{Obozinski-Bach:hal-01412385}, but more motivated by
algorithmic preoccupations.
The second main difference with~\cite{Obozinski-Bach:hal-01412385} relates to
support recovery/identification.
\cite[Proposition~7]{Obozinski-Bach:hal-01412385}
provides a probabilistic ``support recovery'' result in the case of regularized
least-squares minimization.
By contrast, one of our main results, Theorem~\ref{th:support_identification},
is a (deterministic) ``support
identification'' one: it provides a condition
(Equation~\eqref{eq:support_identification_unique_a}) under which the solution
of an optimization problem with generalized $k$-support dual~norm penalty has an optimal
solution with support of cardinality less than~$k$
(Equation~\eqref{eq:support_identification_unique_b}).
\medskip

Finally, \cite{Candes:ICM2014} presents compressive sensing and convex sparsity-promoting methods, 
whereas \cite{Drusvyatskiy-Vavasis-Wolkowicz:2015} studies the geometry of the rank–sparsity ball providing complementary 
	rank–sparsity trade-off results.
\medskip

To summarize, to the difference with the literature mentioned above, 
we focus on sparsity-inducing norms that promote solutions within an a priori ``sparsity budget'' by
using dual information, and our analysis is mostly geometric.

\subsection{Organization of the paper}
\label{Organization_of_the_paper}

The paper is organized as follows.

In Sect.~\ref{Support_identification}, after providing background on sparsity
and on faces of closed convex sets, we define the \emph{Sparse Projection and
  Convexification (SPaC)} method, a systematic way to generate, from a given
(nonempty) set~\( \Primal \),
a closed convex set whose extreme points are $k$-sparse,
called the \emph{$k$-SPaC hull of~\( \Primal \)}.
More precisely, we project the set~\( \Primal \) onto the
different subspaces of $k$-sparse vectors, and then we take the closed convex
closure of this union of projections. As it is well known that this last operation
does not generate extreme points outside this union, we obtain that the resulting 
closed convex set, the $k$-SPaC hull of~\( \Primal \), has extreme points that
are $k$-sparse\footnote{%
This operation --- which consists in \emph{projecting} the set~\( \Primal \)
onto all the subspaces made of $k$-sparse vectors, and then take the closed
convex closure --- is not as natural as \emph{intersecting} the set~\( \Primal \) with all 
$k$-sparse vectors, and then take the closed convex closure.
Both operations lead to closed convex sets whose (different) extreme points are $k$-sparse.
However, the SPaC method is the most adequate to state our main result.}.
Then, we expose our main result: we show that any face of a $k$-SPaC hull, exposed by a given dual vector, is
obtained by projections --- onto a specific selection (depending on the dual vector)
of subspaces of $k$-sparse vectors --- of the convex closure of the original
set.
This specific selection is the basis for support selection
results.

In Sect.~\ref{The_case_of_generalized_top-k_and_k-support_norms}, we recall the
definition of generalized $k$-support dual~norms. Their unit balls appear to be 
produced exactly as above, that is, are $k$-SPaC hulls of a given source unit ball. Thus, we are
able to describe the exposed faces of the unit ball of $k$-support dual~norms as projections,
onto a specific selection of subspaces of $k$-sparse vectors, of the original
unit ball.
As an application we return to the original motivation of the paper as was exposed
in~\S\ref{Motivation}, and we provide dual conditions under which the
primal optimal solution of a minimization problem, penalized by a $k$-support dual~norm, is $k$-sparse.
Then, we recall the notions of orthant-monotonic and orthant-strictly
monotonic norms; for such source norms, we obtain a more precise description of the intersection of the
$k$-sparse vectors with the exposed faces of the resulting $k$-support dual~norm.

In Sect.~\ref{Geometry_of_lptopnorm{q}{k}_and__lpsupportnorm{p}{k}s_unit_balls},
we examine the geometric aspects of the faces and normal cones of the unit balls
of \lpsupportnorm{p}{k}s and \lptopnorm{q}{k}s --- that are the generalized
$k$-support dual~norms of
Sect.~\ref{The_case_of_generalized_top-k_and_k-support_norms},
and their dual norms, in the case where the source norm is one of the $\ell_p$
norms (and $1/p+1/q=1$).
First, we recall known results for the case $p=+\infty$, where the unit
balls are polytopal and therefore all faces are exposed. Second, we present new
results for the case $1<p<+\infty$. In particular, we show that every proper
face, exposed or not, of the unit ball of the $k$-support dual~norms is a
hypersimplex, i.e, the convex hull of $0/1$-valued points with the same
$\lzero$-norm (see~Equation~\eqref{eq:lzeropseudonorm}).

Long proofs are relegated to Appendix~\ref{Proofs}.

\section{Exposed faces of $k$-SPaC hulls}
\label{Support_identification}

In~\S\ref{Background}, we provide background on sparsity and on exposed faces of
closed convex sets.
In~\S\ref{Faces_of_compact_convex_sets_with}, we define $k$-SPaC hulls
(closed convex sets whose extreme points are $k$-sparse),
and then we state our main result: we characterize the exposed faces of $k$-SPaC
hulls. We also deduce support identification.
Proofs are given in Appendix~\ref{Proofs}.

\subsection{Background on sparsity and on exposed faces of closed convex sets}
\label{Background}

Let ${\spacedim} \in \NN^*$ be a natural number.
We consider the finite-dimensional real Euclidean vector space~$\RR^{\spacedim}$
equipped with the scalar product~\( \proscal{\,}{} \).

\subsubsection*{Background on sparsity}
\label{Background_on_sparsity}

We use the notation \( \ic{j,k}=\na{j, j+1,\ldots,k-1,k} \) for any pair of integers
such that \( j \leq k \).
Denoting by $\cardinal{\IndexSubset}$ the cardinality of a subset \( \IndexSubset
\subset \Vset \), we define, for any vector~$\primal$ in $\RR^{\spacedim}$,
the \emph{support} of~\( \primal \) by
\begin{subequations}
  \begin{align}
    \Support{\primal}
    &=
      \bsetco{ j \in \Vset }{\primal_j \not= 0 } \subset \Vset
      \eqsepv
      \intertext{and the  \emph{\lzeropseudonorm} of~\( \primal \) by the number
      of nonzero components, that is, by}
      \lzero\np{\primal}
    &=
      \cardinal{~\Support{\primal}} = \sum_{j \in \Vset }\1_{\primal_j \not= 0 } 
      \eqfinp
      \label{eq:lzeropseudonorm}
  \end{align}
\end{subequations}
\begin{subequations}
  This defines the \emph{\lzeropseudonorm} function
  \( \lzero : \RR^{\spacedim} \to \Vsetzero \). For any $k$ in $\Vset$, we
  denote its level sets, made of all the vectors with at most~$k$ nonzero
  coordinates, by
  \begin{equation}
    \LevelSet{\lzero}{k} = \nsetco{ \primal \in \RR^{\spacedim}}{\lzero\np{\primal} \leq k}
  \end{equation}  
  The vectors in \( \LevelSet{\lzero}{k} \) will be called \emph{$k$-sparse
    vectors}
  (so they can have exactly $k$~nonzero entries, but they can also have strictly
  less than $k$~nonzero entries).
  For any subset \( \IndexSubset\) of \( \Vset \), we introduce 
  the subspace $\FlatRR_{\IndexSubset}$ of~$\RR^{\spacedim}$ made of vectors
  whose components vanish outside of~$\IndexSubset$ as
  \begin{equation}
    \label{eq:FlatRR}
    \FlatRR_{\IndexSubset} =
    \bsetco{ \primal \in \RR^{\spacedim} }%
    { \primal_j=0 \eqsepv \forall j \not\in \IndexSubset }
  \end{equation} 
  with the convention that \( \FlatRR_{\emptyset}= \na{0} \).
  Using \( \bigcup_{\cardinal{\IndexSubset}\leq k} \) as a shorthand for 
  \( \bigcup_{ \IndexSubset \subset\Vset, \cardinal{\IndexSubset}\leq k} \),
  we get
  \begin{equation} 
    {\LevelSet{\lzero}{k}}   =
    \bigcup_{\cardinal{\IndexSubset}\leq k}\FlatRR_{\IndexSubset}
    \eqfinp
    \label{eq:LevelSet_lzero_FlatRR_IndexSubset}
  \end{equation}
\end{subequations}

We denote by
\( \ProjectionIndexSubset : \RR^{\spacedim} \to \FlatRR_{\IndexSubset} \) the
\emph{orthogonal projection mapping}; for any vector \( \primal \) in
\( \RR^{\spacedim} \), the coordinates of the vector
\( \ProjectionIndexSubset\primal \in \FlatRR_{\IndexSubset} \) coincide with those
of~\( \primal \), except for the ones whose indices range outside
of~$\IndexSubset$ that are equal to zero.
The orthogonal projection
mapping~$\ProjectionIndexSubset$ is \emph{self-adjoint} (or \emph{self-dual}),
that is,
\begin{equation}
  \proscal{\ProjectionIndexSubset\primal}{\dual}
  =
  \proscal{\primal}{\ProjectionIndexSubset\dual}
  =
  \proscal{\ProjectionIndexSubset\primal}{\ProjectionIndexSubset\dual}
  \eqsepv 
  \forall \primal \in \RR^{\spacedim} 
  \eqsepv 
  \forall \dual \in \RR^{\spacedim} 
  \eqfinp
  \label{eq:orthogonal_projection_self-dual}
\end{equation}

\subsubsection*{Background on exposed faces of closed convex sets}

For any subset \( \Primal\subset\PRIMAL \), the expression 
\begin{equation}
  \SupportFunction{\Primal}\np{\dual} = 
  \sup_{\primal\in\Primal} \proscal{\primal}{\dual}
  \eqsepv \forall \dual \in \DUAL
  \label{eq:support_function}
\end{equation}
defines a map \( \SupportFunction{\Primal} : \DUAL \to \barRR\) called the 
\emph{support function\footnote{%
    Note that the support \emph{function} has nothing to do with the support of a \emph{vector}.}
  of the subset~$\Primal$}.
The (negative) \emph{polar set}~$\PolarSet{\Primal}$ 
of the subset~$\Primal \subset \PRIMAL$ is the closed convex set
\begin{equation}
  \PolarSet{\Primal}=
  \bsetco{ \dual \in \DUAL }{  \nscal{\primal}{\dual} \leq 1
    \eqsepv \forall \primal \in \Primal }
  =  \MidLevelSet{\SupportFunction{\Primal}}{1}
  \eqfinp
\label{eq:PolarSet}
\end{equation}

The face of a nonempty closed convex subset $\Convex$ of $\PRIMAL$ exposed by a
dual vector~$\dual$ in $\DUAL$ is
\begin{equation}
  \ExposedFace(\Convex,\dual) 
  =\argmax_{\primal\in\Convex} \proscal{\primal}{\dual}\eqfinv
  \label{eq:ExposedFace}
\end{equation}
and the normal cone~$\NormalCone(\Convex,\primal)$ of~$\Convex$ at a primal
vector~\( \primal\in\Convex \) is defined by the conjugacy relation
\begin{equation}
  \primal \in \Convex \mtext{ and } \dual \in \NormalCone(\Convex,\primal) 
  \iff 
  \primal \in \ExposedFace(\Convex,\dual) 
  \eqfinp
\end{equation}

\subsection{Sparse projection and convexification}
\label{Faces_of_compact_convex_sets_with}

As discussed in the introduction, the intuition behind
\cite[Figure~2]{Tibshirani:1996} is that the unit ball of a sparsity-inducing
norm should have extreme points (vertices) precisely at $k$-sparse vectors.  One
way to enforce this is to select a suitable subset
of $k$-sparse vectors, and then take the convex closure,
as described in Definition~\ref{de:SPaC}
of what we call \emph{Sparse Projection and Convexification}.
Then, Theorem~\ref{th:Support_identification} characterizes the exposed faces
of this convex closure.

\begin{definition}(Sparse Projection and Convexification)
  \label{de:SPaC}
  Let \( k\in\Vset \) be a natural number and
  \( \Primal \subset \PRIMAL\) be a (primal) nonempty set.
  The \emph{Sparse Projection and Convexification (SPaC)} method
  consists in projecting the set~\( \Primal \) onto all $k$-sparse subspaces~\(
  \FlatRR_{\IndexSubset} \) in~\eqref{eq:FlatRR} with \(
  \cardinal{\IndexSubset}\leq k \), and obtain the following set~$\AtomicSet$
  made of $k$-sparse vectors, that is,
  \begin{subequations}
    \begin{equation}
    \AtomicSet 
    = \bigcup_{\cardinal{\IndexSubset}\leq k}
    \ProjectionIndexSubset\np{\Primal}
    \subset  {\LevelSet{\lzero}{k}} 
    \eqfinv 
    \label{eq:sparse_atomic_set} 
  \end{equation}
  and its closed convex hull, called the \emph{$k$-SPaC hull of~\( \Primal \)}, 
    \begin{equation}
    \closedconvexhull\AtomicSet 
    = \closedconvexhull \bp{\bigcup_{\cardinal{\IndexSubset}\leq k}
      \ProjectionIndexSubset\np{\Primal} }
    \eqfinp
    \label{eq:SPaC_hull}
  \end{equation}    
  \end{subequations}
\end{definition}

Theorem~\ref{th:Support_identification} below characterizes the exposed faces
of the $k$-SPaC hull.

\begin{theorem}[Characterization of exposed faces of $k$-SPaC hulls]
  \label{th:Support_identification}
  Let \( \dual \in \DUAL \) be a (dual) vector.
  We set\footnote{%
    See Footnote~\ref{ft:optimal_solution} for the notation~$\opt$.}
  \begin{equation}
    \OptimalSupports{\Primal,k}{\dual} =
    \argmax_{\substack{\IndexSubset \subset\Vset \\ \cardinal{\IndexSubset}\leq k}}   
    \SupportFunction{\closedconvexhull\Primal}\np{\ProjectionIndexSubset\dual}
    \eqfinv
    \label{eq:OptimalSupports}
  \end{equation}
which is the nonempty subset of all subsets \( \ba{ \IndexSubset \subset\Vset,
  \cardinal{\IndexSubset}\leq k} \) of indices that achieve the maximum of
\( \IndexSubset \mapsto
\SupportFunction{\closedconvexhull\Primal}\np{\ProjectionIndexSubset\dual} \).
  %
  Then, 
  the exposed faces of the $k$-SPaC hull~$\closedconvexhull\AtomicSet$ in~\eqref{eq:SPaC_hull}
  are related to the exposed faces of~$\closedconvexhull\Primal$ by 
  \begin{equation}
    \label{eq:Support_identification}
    \AtomicSet
    \cap
    \ExposedFace\bp{\overbrace{\closedconvexhull\AtomicSet}^{\substack{\textrm{new closed}\\\textrm{convex set:}\\\textrm{$k$-SPaC hull}}},\dual}
    = 
    \Bsetco{\underbrace{\ProjectionIndexSubsetOpt}_{\textrm{projection}}\bp{\Primal \cap
        \ExposedFace(\overbrace{\closedconvexhull\Primal}^{\substack{\textrm{original
              closed}\\\textrm{convex set}}},
        \underbrace{\ProjectionIndexSubsetOpt\dual}_{\textrm{projection}})}}
    {\underbrace{\IndexSubset\opt \in
        \OptimalSupports{\Primal,k}{\dual}}_{\textrm{projection selection}}}
    \eqfinv
  \end{equation}
  and by
  \begin{equation}
    \label{eq:Face_description}
\underbrace{
  \ExposedFace\bp{\overbrace{\closedconvexhull\AtomicSet}^{\substack{\textrm{new
          closed}\\\textrm{convex set:}\\\textrm{$k$-SPaC hull}}},\dual}
}_{\textrm{exposed face of~$\closedconvexhull\AtomicSet$}}
    = 
    \closedconvexhull
    \Bsetco{\underbrace{\ProjectionIndexSubsetOpt}_{\textrm{projection}}\bp{\Primal \cap
        \underbrace{
          \ExposedFace(\overbrace{\closedconvexhull\Primal}^{\substack{\textrm{original
              closed}\\\textrm{convex set}}},
        \underbrace{\ProjectionIndexSubsetOpt\dual}_{\textrm{projection}})}_{\textrm{exposed
          face of~$\closedconvexhull\Primal$}}}
    }
    {\underbrace{\IndexSubset\opt \in
        \OptimalSupports{\Primal,k}{\dual}}_{\textrm{projection selection}}}
    \eqfinp
  \end{equation}
%
\end{theorem}
By construction, the extreme points of \(
\closedconvexhull\AtomicSet \) are $k$-sparse vectors:
indeed, the extreme points of \( \closedconvexhull\AtomicSet \) are contained in \(
\AtomicSet \), while \(
\AtomicSet \) itself is a subset of \( \bigcup_{\cardinal{\IndexSubset}\leq k}
\FlatRR_{\IndexSubset} =  {\LevelSet{\lzero}{k}} \) (see
Equation~\eqref{eq:sparse_atomic_set}), hence is made of 
$k$-sparse vectors. 

As a corollary of Theorem~\ref{th:Support_identification}, we obtain support identification as follows.

\begin{corollary}[Support identification]
  \label{cor:Support_identification}
  Under the assumptions of Theorem~\ref{th:Support_identification},
  \begin{subequations}
  \begin{align}
    \primal \in \AtomicSet \cap \ExposedFace(\closedconvexhull\AtomicSet,\dual)    
    & \implies
      \Support{\primal} \in \OptimalSupports{\Primal,k}{\dual}
      \eqfinv
      \label{eq:Support_identification_Support_atom_opt}
    \\
    \primal \in \ExposedFace(\closedconvexhull\AtomicSet,\dual)
    & \implies
      \Support{\primal} \subset \bigcup_{ \IndexSubset\opt \in
      \OptimalSupports{\Primal,k}{\dual} }
      \IndexSubset\opt
      \eqfinp 
      \label{eq:Support_identification_Support_primal_opt}
  \end{align}    
  \end{subequations}
\end{corollary}

\section{Exposed faces of unit balls of 
  $k$-support dual~norms}
\label{The_case_of_generalized_top-k_and_k-support_norms}

In~\S\ref{Geometry_of_the_unit_balls_of_generalized_top-k_and_k-support_dual_norms},
  we provide background on generalized top-$k$ and $k$-support dual~norms.
In~\S\ref{Face_characterization_and_support_identification}, we apply
Theorem~\ref{th:Support_identification} with (primal) set the unit ball
of a norm, and obtain thus a characterization of the exposed faces of the unit
ball of $k$-support dual~norms.
In~\S\ref{Finding_optimal_solutions_with_at_most_k_nonzero_coordinates}, we
return to the original motivation of the paper as was exposed
in~\S\ref{Motivation}: we provide dual conditions under which the
primal optimal solution of a minimization problem, penalized by a $k$-support dual~norm, is $k$-sparse.
In~\S\ref{The_orthant-monotonic_case}, we recall the notion of orthant-monotonic
norm and, in this case, we obtain a simpler characterization of the 
$k$-sparse vectors in the exposed faces of the unit
ball of $k$-support dual~norms.
Finally, in~\S\ref{The_orthant-strictly_monotonic_case}, we recall the notion of
orthant-strictly monotonic norm and, in this case, we obtain an even simpler characterization of the 
$k$-sparse vectors in the exposed faces of the unit
ball of $k$-support dual~norms.
\medskip

We recall that ${\spacedim} \in \NN^*$ is a natural number, and that 
we consider the finite-dimensional real Euclidean vector space~$\RR^{\spacedim}$
equipped with the scalar product~\( \proscal{\,}{} \).

\subsection{Background on generalized top-$k$ and $k$-support dual~norms}
\label{Geometry_of_the_unit_balls_of_generalized_top-k_and_k-support_dual_norms}

We provide background on generalized top-$k$ and $k$-support dual~norms
that are constructed by means of a source norm.
These families of norms have been introduced
in~\cite{Chancelier-DeLara:2022_SVVA} in relation to hidden convexity in the
\lzeropseudonorm.
In what follows, we will observe that the unit ball of the generalized $k$-support dual~norm is exactly
the $k$-SPaC of the unit ball of the source norm, as introduced in Definition~\ref{de:SPaC}.

In the following, the symbol~$\star$ in the superscript indicates that the generalized
$k$-support dual~norm \( \SupportDualNorm{\TripleNorm{\cdot}}{k} \) is the dual norm
of the generalized top-$k$ dual norm \( \TopDualNorm{\TripleNorm{\dual}}{k} \)
and, thus, is a norm on the primal space. To stress the point, we use~$\primal$
for a primal vector, like in~\( \SupportDualNorm{\TripleNorm{\primal}}{k} \),
and~$\dual$ for a dual vector, like
in~\( \TopDualNorm{\TripleNorm{\dual}}{k} \).
\medskip

The following recalls can also be found in
Table~\ref{tab:generalized_support_norms}
in~\S\ref{Summary_table_on_generalized_top-k_and_k-support_norms}, 
which provides a summary of generalized top-$k$ and $k$-support norms and dual~norms
that appeared in \cite{Chancelier-DeLara:2022_OSM_JCA,
  Chancelier-DeLara:2022_SVVA}.

\begin{definition}[\cite{Chancelier-DeLara:2022_SVVA}]
  \label{de:top_norm}
  Let $\TripleNorm{\cdot}$ be a (source) norm on~$ \RR^{\spacedim}$, with unit ball~$\TripleNormBall$.
  The unit
  ball~\( \TripleNormDualBall \) of the dual norm~\( \TripleNormDual{\cdot} \) is
  the polar set of~$\TripleNormBall$, that is,
  \( \TripleNormDualBall = \PolarSet{\TripleNormBall} \) as defined
  in~\eqref{eq:PolarSet}.
  For any \( k \in \Vset \), and using \(\sup_{\cardinal{\IndexSubset} \leq k} \) as a
  shorthand for
  \( \sup_{ { \IndexSubset \subset \ic{1,d}, \cardinal{\IndexSubset} \leq k}}\), we call
  \begin{itemize}
  \item[$(i)$] 
    \emph{generalized top-$k$ dual~norm} 
    the norm defined by~\cite[Eq.~(10)]{Chancelier-DeLara:2022_SVVA}
    \begin{equation}
      \TopDualNorm{\TripleNorm{\dual}}{k}
      =
      \underbrace{  \sup_{\cardinal{\IndexSubset} \leq k}
        \TripleNormDual{\overbrace{\ProjectionIndexSubset\np{\dual}}%
          ^{\substack{\textrm{$k$-sparse}\\ \textrm{projection} \\ \textrm{on $\FlatRR_{\IndexSubset}$}}}}
      }_{\substack{\textrm{exploring all}\\ \textrm{$k$-sparse projections}\\
        \textrm{and keeping the one} \\ \textrm{with largest dual norm}}}
      \eqsepv \forall \dual \in \RR^{\spacedim}
      \eqfinv 
    \end{equation}
    whose unit ball is 
    \begin{equation}
      \TopDualNorm{\TripleNormBall}{k} 
      =    \bsetco{\dual \in \RR^{\spacedim}}{\TopDualNorm{\TripleNorm{\dual}}{k} \leq 1} 
      =  \bigcap_{\cardinal{\IndexSubset} \leq k}
       \underbrace{ \Converse{\ProjectionIndexSubset} \np{\FlatRR_{\IndexSubset}
           \cap \TripleNormDualBall} }_{\substack{\textrm{cylinder}}}
      \eqfinv
      \label{eq:generalized_top-k_dual_norm_unit_ball}
    \end{equation}
    hence is an intersection of cylinders, 
  \item[$(ii)$] 
    \emph{generalized $k$-support dual~norm} the corresponding dual norm (of the generalized top-$k$ dual~norm)
    \cite[Eq.~(11)]{Chancelier-DeLara:2022_SVVA}
    \begin{equation}
      \SupportDualNorm{\TripleNorm{\cdot}}{k} = \bp{
        \TopDualNorm{\TripleNorm{\cdot}}{k} }_{\star} 
      \eqfinv
      \label{eq:generalized_k-support_dual_norm}
    \end{equation}
    whose unit ball is 
    \begin{equation}
      \SupportDualNorm{\TripleNormBall}{k} = 
      \defsetco{\primal \in \RR^{\spacedim}}{\SupportDualNorm{\TripleNorm{\primal}}{k} \leq 1} 
      = \closedconvexhull\bp{ \bigcup_{ \cardinal{\IndexSubset} \leq k} 
        \ProjectionIndexSubset\np{ \TripleNormBall } }
      = \closedconvexhull\bp{ \bigcup_{ \cardinal{\IndexSubset} \leq k} 
        \ProjectionIndexSubset\np{
          \TripleNormSphere } }
      \eqfinv
      \label{eq:generalized_k-support_dual_norm_unit_ball}
    \end{equation}
    and whose unit sphere is denoted by~\( \SupportNorm{\TripleNormSphere}{k} \).
  \end{itemize}
\end{definition}

So, it appears that the unit ball~\( \SupportDualNorm{\TripleNormBall}{k}\) of the generalized $k$-support dual~norm, as
expressed in~\eqref{eq:generalized_k-support_dual_norm_unit_ball}, is exactly
the $k$-SPaC of the unit ball~$\TripleNormBall$ of the source
norm~$\TripleNorm{\cdot}$
as expressed in~\eqref{eq:SPaC_hull} (it is also the $k$-SPaC hull of the unit
sphere~$\TripleNormSphere$).

\subsection{Exposed faces of unit balls of $k$-support dual~norms}
\label{Face_characterization_and_support_identification}

Here, we apply the result of Theorem~\ref{th:Support_identification} with
(primal) set the unit ball of a norm.

\begin{proposition}
  \label{pr:Support_identification_generalized_k-support_norm}
  Let $\TripleNorm{\cdot}$ be a norm on~$ \RR^{\spacedim}$, that we call the source norm,
  with unit ball~$\TripleNormBall$.
Then, for any \( k \in \Vset \) and any dual vector~$\dual \in \DUAL$, 
the exposed faces of the unit ball~$\SupportDualNorm{\TripleNormBall}{k}$
  of the generalized $k$-support dual~norm are related to the exposed faces
of the unit ball~$\TripleNormBall$ of the source norm by   
  \begin{equation}
    \overbrace{
      \ExposedFace\np{\SupportDualNorm{\TripleNormBall}{k},
        \underbrace{\dual}_{\substack{\textrm{dual} \\ \textrm{vector}}}
      }
}^{{\substack{\textrm{exposed face of} \\ \textrm{the unit
        ball~$\SupportDualNorm{\TripleNormBall}{k}$} \\ \textrm{of the generalized} \\ \textrm{$k$-support dual~norm} }}}
      =
    \closedconvexhull
    \Bsetco{
      \underbrace{\ProjectionIndexSubsetOpt}%
      _{\substack{\textrm{$k$-sparse}\\ \textrm{projection} \\ \textrm{on $\FlatRR_{\IndexSubset\opt}$}}}
      \bp{
        \overbrace{ \ExposedFace\np{\TripleNormBall,
            \underbrace{\ProjectionIndexSubset\opt\dual}%
            _{\substack{\textrm{$k$-sparse}\\ \textrm{projection} \\
          \textrm{of the} \\ \textrm{dual vector} \\ \textrm{on
            $\FlatRR_{\IndexSubset\opt}$} } }
    }
       }^{{\substack{\textrm{exposed face of} \\ \textrm{the unit
        ball~$\TripleNormBall$} \\ \textrm{of the source norm} }}}
     }}
    {\underbrace{\IndexSubset\opt \in 
      \argmax_{\cardinal{\IndexSubset}\leq k}
      \TripleNormDual{\ProjectionIndexSubset\dual}}%
    _{\substack{\textrm{selection of} \\ \textrm{$k$-sparse}\\
        \textrm{projection spaces}  \\ \textrm{$\FlatRR_{\IndexSubset\opt}$}}}
  }
    \eqfinp
    \label{eq:Face_description_generalized_k-support_norm}
  \end{equation}
  %
\end{proposition}

\begin{proof}
  The proof results from Theorem~\ref{th:Support_identification} with
  \( \Primal= \TripleNormBall \),
  \(  \AtomicSet = \bigcup_{\cardinal{\IndexSubset}\leq k}
  \ProjectionIndexSubset\np{\Primal}
  = \bigcup_{\cardinal{\IndexSubset}\leq k}
  \ProjectionIndexSubset\np{\TripleNormBall} \)
  in~\eqref{eq:sparse_atomic_set}, and
  \[
    \OptimalSupports{\Primal,k}{\dual}
    =
    \argmax_{\substack{\IndexSubset \subset\Vset \\ \cardinal{\IndexSubset}\leq k}}   
    \SupportFunction{\TripleNormBall}\np{\ProjectionIndexSubset\dual}
    =
    \argmax_{\substack{\IndexSubset \subset\Vset \\ \cardinal{\IndexSubset}\leq k}}   
    \SupportFunction{\closedconvexhull\Primal}\np{\ProjectionIndexSubset\dual}
    =
    \argmax_{\cardinal{\IndexSubset}\leq k}
    \TripleNormDual{\ProjectionIndexSubset\dual}
  \]
  in~\eqref{eq:OptimalSupports}, as \(  \SupportFunction{\TripleNormBall} \) is
  the dual norm~$\TripleNormDual{\cdot}$.
  Then --- using the
  expression~\eqref{eq:generalized_k-support_dual_norm_unit_ball} of
  \( \SupportDualNorm{\TripleNormBall}{k} = \closedconvexhull\bp{ \bigcup_{
      \cardinal{\IndexSubset} \leq k} \ProjectionIndexSubset\np{ \TripleNormBall }
  } \) --- Equation~\eqref{eq:Support_identification} gives
  \begin{equation}
    \label{eq:Support_identification_generalized_k-support_norm}
    \bp{ \bigcup_{ \cardinal{\IndexSubset} \leq k} 
      \ProjectionIndexSubset\np{ \TripleNormBall } }
    \cap    \ExposedFace\bp{\SupportDualNorm{\TripleNormBall}{k},\dual}
    = 
    \Bsetco{\ProjectionIndexSubsetOpt\bp{
        \ExposedFace(\TripleNormBall,\ProjectionIndexSubsetOpt\dual)}}
    {\IndexSubset\opt \in
       \argmax_{\cardinal{\IndexSubset}\leq k} \TripleNormDual{\ProjectionIndexSubset\dual}}
    \eqfinv
  \end{equation}
and Equation~\eqref{eq:Face_description} 
gives
  Equation~\eqref{eq:Face_description_generalized_k-support_norm}.
\end{proof}

\subsection{Finding optimal solutions with at most~$k$ nonzero coordinates}
\label{Finding_optimal_solutions_with_at_most_k_nonzero_coordinates}

Here, we make the connection with the original motivation as exposed in~\S\ref{Motivation}.
To obtain sparse solution to a minimization problem, the standard approach --- 
which consists in adding an ($\ell_1$)-norm penalty (Lasso) --- does not provide
a direct mechanism to control the number of nonzero entries in the solution.  
We now show that, by adding a generalized top-$k$ dual~norm, we can control this number under a
condition provided below.

\begin{theorem}
  \label{th:support_identification}
  Let \( \fonctionprimal : \RR^{\spacedim} \to \RR \) be a smooth convex function, and
  $\gamma >0$.
  Let $\TripleNorm{\cdot}$ be a norm on~$ \RR^{\spacedim}$. 
  For given {sparsity threshold~\( k \in \ic{1,{\spacedim}} \)},
  we consider the {generalized top-$k$ dual~norm}~\(
  \SupportDualNorm{\DoubleNorm{\cdot}}{{k}} \) (see
  Definition~\ref{de:top_norm}).
  Then, an {optimal solution~\(  \primal\opt \)} of 
  \begin{subequations}
    \begin{equation}
      \min_{\primal\in\RR^{\spacedim}}
      \bp{\fonctionprimal\np{\primal} + \gamma
        \SupportDualNorm{\DoubleNorm{\primal}}{{k}}  }
    \end{equation}
    has support
    \begin{equation}
      {  \Support{\atom\opt} } \subset
      \bigcup_{\substack{\IndexSubset\opt
          \in \argmax_{\cardinal{\IndexSubset}\leq {k}} \\
          \DoubleNormDual{\pi_{\IndexSubset}\np{-\nabla\fonctionprimal\np{\primal\opt}}} } }
      \IndexSubset\opt
      \eqfinp
      \label{eq:support_identification_b}
    \end{equation}
    \label{eq:support_identification}
  \end{subequations}
  As a consequence, if
  \begin{subequations}
    \begin{equation}
      \argmax_{\cardinal{\IndexSubset}\leq {k}}
      \DoubleNormDual{\pi_{\IndexSubset}\np{-\nabla\fonctionprimal\np{\primal\opt}}}
      = \IndexSubset\opt \quad  {\textrm{is unique}}
      \eqfinv 
     \label{eq:support_identification_unique_a}
   \end{equation}
    \begin{equation}
      \textrm{then }\quad   { \Support{\atom\opt} \subset     \IndexSubset\opt }
      \text{ with } { \cardinal{\IndexSubset\opt}\leq k }
      \eqfinv 
   \label{eq:support_identification_unique_b}
     \end{equation}
    \label{eq:support_identification_unique}
  \end{subequations}
  so that  the {optimal solution~\( \primal\opt \)} is {$k$-sparse}.
\end{theorem}

\begin{proof}
  We have that 
  \begin{align*}
    &
      \primal\opt \in \argmin_{\primal\in\RR^{\spacedim}}
      \bp{\fonctionprimal\np{\primal} + \gamma
      \SupportDualNorm{\DoubleNorm{\primal}}{{k}} }
    \\
    \iff
    &
      \primal\opt \in \argmin_{\primal\in\RR^{\spacedim}}
      \bp{\fonctionprimal\np{\primal} + \gamma
      \SupportFunction{\TopDualNorm{\TripleNormBall}{k}}  }
      \tag{by~\eqref{eq:generalized_top-k_dual_norm_unit_ball}}      
    \\
    \iff
    &
      0 \in \partial \bp{ \fonctionprimal
      + \gamma \SupportFunction{\TopDualNorm{\TripleNormBall}{k}}}\np{\primal\opt}
      \tag{by the Fermat rule}
    \\
    \iff
    &
      0 \in \partial {\fonctionprimal}\np{\primal\opt}
      + \gamma
      \partial\SupportFunction{\TopDualNorm{\TripleNormBall}{k}}\np{\primal\opt}
      \intertext{by \cite[Corollary~16.48]{Bauschke-Combettes:2017}, as both
      functions are proper convex lsc, and $\dom\fonctionprimal=\RR^{\spacedim}$,}
      \iff
    &
      0 \in \nabla{\fonctionprimal}\np{\primal\opt}
      + \gamma
      \partial\SupportFunction{\TopDualNorm{\TripleNormBall}{k}}\np{\primal\opt}
      \tag{by \cite[Proposition~17.31~(i)]{Bauschke-Combettes:2017}}
    \\
    \iff
    &
      0 \in \nabla{\fonctionprimal}\np{\primal\opt}
      + \gamma \ExposedFace(\TopDualNorm{\TripleNormBall}{k},\primal\opt)
      \intertext{as the subdifferential of a support function is the support
      function of the corresponding face,           see for instance
      \cite[Theorem~1.7.2]{Schneider:2014}, \cite[Corollary~8.25]{Rockafellar-Wets:1998},}
      \iff
    &
      -\nabla\fonctionprimal\np{\primal\opt} \in
      \gamma \ExposedFace(\TopDualNorm{\TripleNormBall}{k},\primal\opt)
    \\
    \iff
    &
      \begin{cases}
        \text{either } \primal\opt=0 \eqsepv
        -\nabla\fonctionprimal\np{0}\in\TopDualNorm{\TripleNormBall}{k}
        \\
        \text{or } \primal\opt\not=0 \eqsepv
        \frac{\primal\opt}{\SupportDualNorm{\TripleNorm{\primal\opt}}{k}} \in
        \ExposedFace\np{\SupportDualNorm{\TripleNormBall}{k},-\nabla\fonctionprimal\np{\primal\opt}}
      \end{cases}
      \tag{by polarity~\cite[Theorem~5.1]{Fan-Jeong-Sun-Friedlander:2020}}
    \\
    &\implies
      \begin{cases}
        \text{either } \primal\opt=0 \eqsepv
        \Support{\atom\opt} = \emptyset 
        \\
        \text{or } \primal\opt\not=0 \eqsepv
        \Support{\atom\opt} \subset \bigcup_{\substack{\IndexSubset\opt
        \in \argmax_{\cardinal{\IndexSubset}\leq {k}} \\
        \DoubleNormDual{\pi_{\IndexSubset}\np{-\nabla\fonctionprimal\np{\primal\opt}}} } }
        \IndexSubset\opt
      \end{cases}
  \end{align*}
by Equation~\eqref{eq:Support_identification_Support_primal_opt} in the support
identification Corollary~\ref{cor:Support_identification},
with \( \Primal = \TripleNormBall \) the unit ball of the
norm~$\TripleNorm{\cdot}$,
and \(  \AtomicSet = \bigcup_{\cardinal{\IndexSubset}\leq k}
  \ProjectionIndexSubset\np{\Primal}
  = \bigcup_{\cardinal{\IndexSubset}\leq k}
  \ProjectionIndexSubset\np{\TripleNormBall} \) in~\eqref{eq:sparse_atomic_set},
and \( \closedconvexhull\AtomicSet = \SupportDualNorm{\TripleNormBall}{k} \)
by~\eqref{eq:generalized_k-support_dual_norm_unit_ball}.
  
  Thus, we have proven~\eqref{eq:support_identification_b}.
  Equation~\eqref{eq:support_identification_unique} follows trivially.
\end{proof}

Recall that a norm \( \TripleNorm{\cdot}\) on~\( \RR^{\spacedim} \) is called
    \emph{monotonic}~\cite{Bauer-Stoer-Witzgall:1961}
    if, for all 
    $\primal$, $\primal'$ in~$ \RR^{\spacedim}$, we have 
    \(
    |\primal| \le |\primal'| \Rightarrow 
    \TripleNorm{\primal} \le \TripleNorm{\primal'}
    \),
    where $|\primal| \leq |\primal'|$ means 
    $|\primal_i| \leq |\primal'_i|$ for all $i\in\ic{1,\spacedim}$.
If, in Theorem~\ref{th:support_identification},
we suppose that the source
  norm~$\TripleNorm{\cdot}$ is monotonic and symmetric, we get
  the following additional result: in~\eqref{eq:support_identification_b},
\( \argmax_{\cardinal{\IndexSubset}\leq {k}} 
          \DoubleNormDual{\pi_{\IndexSubset}\np{-\nabla\fonctionprimal\np{\primal\opt}}} \) can be obtained by
       ordering the components of the vector \( \module{\nabla\fonctionprimal\np{\primal\opt}} \) and then
       taking indices of the~$k$ largest entries.
       Indeed, if $\TripleNorm{\cdot}$ is monotonic, so is its
       dual norm~\( \TripleNormDual{\cdot} \) as shown in~\cite{Bauer-Stoer-Witzgall:1961};
if $\TripleNorm{\cdot}$ is symmetric, so is its
       dual norm~\( \TripleNormDual{\cdot} \) (straightforward proof). The claimed result follows.
\bigskip

We conclude this part with the following Corollary, certainly well known in the Lasso literature.
\begin{corollary}
  Let \( \fonctionprimal : \RR^{\spacedim} \to \RR \) be a smooth convex function,
  $\gamma >0$ and $\DoubleNorm{\cdot}_1$ be the $\ell_1$-norm.  An {optimal
    solution~\( \primal\opt \)} of
  \begin{subequations}
    \begin{equation}
      \min_{\primal\in\RR^{\spacedim}}
      \bp{\fonctionprimal\np{\primal} + \gamma\DoubleNorm{\primal}_1 }
    \end{equation}
    has support
    \begin{equation}
      {  \Support{\atom\opt} \subset
        \argmax_{j \in \ic{1,{\spacedim}} }
        \module{\nabla_{j}\fonctionprimal\np{\primal\opt}} }
      \eqfinp 
      \label{eq:support_identification_corollary}
    \end{equation}  
  \end{subequations}
\end{corollary}

\begin{proof}
  If $\TripleNorm{\cdot}$ is the $\ell_1$-norm $\TripleNorm{\cdot}_{1}$
  on~$ \RR^{\spacedim}$, then the {generalized top-$k$
    dual~norm}~\( \SupportDualNorm{\DoubleNorm{\cdot}}{{k}} \) is also the
  $\ell_1$-norm $\TripleNorm{\cdot}_{1}$, for any \( k \in \ic{1,{\spacedim}} \) (see
  Table~\ref{tab:Examples_of_generalized_top-k_and_k-support_dual_norms}).
  Then, we apply Theorem~\ref{th:support_identification} in the case~$k=1$, and
  we get~\eqref{eq:support_identification_b}, which is
  exactly~\eqref{eq:support_identification_corollary} as
  \( \bigcup_{\substack{\IndexSubset\opt
      \in \argmax_{\cardinal{\IndexSubset}\leq {1}} \\
      \DoubleNormDual{\pi_{\IndexSubset}\np{-\nabla\fonctionprimal\np{\primal\opt}}} }
  } \IndexSubset\opt = \argmax_{j \in \ic{1,{\spacedim}} }
  \module{\nabla_{j}\fonctionprimal\np{\primal\opt}} \).
\end{proof}

\subsection{The case of orthant-monotonic source norms}
\label{The_orthant-monotonic_case}

The notion of \emph{orthant-monotonic norm}\footnote{%
  A norm is orthant-monotonic if and only if it is monotonic in every orthant,
  see~\cite[Lemma~2.12]{Gries:1967}, hence the name.} has been introduced
in~\cite{Gries:1967,Gries-Stoer:1967} and an equivalent characterization is
provided in~\cite[Item~7 in Proposition~4]{Chancelier-DeLara:2022_OSM_JCA}: a
norm~$\TripleNorm{\cdot}$ on~\( \RR^{\spacedim} \) is orthant-monotonic if and only if it is \emph{increasing
  with the coordinate subspaces}, in the sense that
$ \TripleNorm{\primal_{J}} \leq \TripleNorm{\primal_\IndexSubset}$ for any
\( \primal \in \RR^{\spacedim} \) and any two subsets $J$ and $K$ of $\Vset$
satisfying \( J \subset \IndexSubset \).  In fact, this is equivalent to
$\TripleNorm{\primal_{J}} \leq \TripleNorm{\primal}$ for any vector
\( \primal \) in \( \RR^{\spacedim} \) and any subset \( J \) of \( \Vset \).

\begin{proposition}
  \label{pr:Support_identification_generalized_k-support_norm_orthant-monotonic}
   Let $\TripleNorm{\cdot}$ be a (source) norm on~$ \RR^{\spacedim}$, with unit ball~$\TripleNormBall$.
    For given {sparsity threshold~\( k \in \ic{1,{\spacedim}} \)},
    we consider the {generalized top-$k$ dual~norm}~\(
    \SupportDualNorm{\DoubleNorm{\cdot}}{{k}} \) (see
    Definition~\ref{de:top_norm}).
  If the source
  norm~$\TripleNorm{\cdot}$ 
  is orthant-monotonic, 
  then for 
  any nonzero dual vector~$\dual \in \DUAL\setminus\na{0}$, we have that

 \begin{equation}
    \label{eq:Support_identification_generalized_k-support_norm_orthant-monotonic}
    \overbrace{
    {\LevelSet{\lzero}{k}} 
    \cap \ExposedFace\bp{ 
        \SupportDualNorm{\TripleNormBall}{k}
        ,
        \underbrace{\dual}%
        _{\substack{\textrm{dual} \\ \textrm{vector}}}
      }
    }^{{\substack{\textrm{$k$-sparse vectors of the exposed face} \\ \textrm{of the unit
        ball~$\SupportDualNorm{\TripleNormBall}{k}$ of} \\ \textrm{the
        generalized $k$-support dual~norm} }}}
    = 
    \Bsetco{ \underbrace{\ProjectionIndexSubsetOpt}%
     _{\substack{\textrm{$k$-sparse}\\ \textrm{projection} \\ \textrm{on $\FlatRR_{\IndexSubset\opt}$}}}
      \bp{ \overbrace{ 
        \ExposedFace\np{ 
          \TripleNormBall
          ,
        \underbrace{\ProjectionIndexSubsetOpt\dual}%
            _{\substack{\textrm{$k$-sparse}\\ \textrm{projection} \\
          \textrm{of the} \\ \textrm{dual vector} \\ \textrm{on
            $\FlatRR_{\IndexSubset\opt}$} } }
    }
           }^{{\substack{\textrm{exposed face of} \\ \textrm{the unit
        ball~$\TripleNormBall$} \\ \textrm{of the source norm} }}}
      }
    }
    {\underbrace{\IndexSubset\opt \in
       \argmax_{\cardinal{\IndexSubset}\leq k}
       \TripleNormDual{\ProjectionIndexSubset\dual}}%
         _{\substack{\textrm{selection of} \\ \textrm{$k$-sparse}\\
        \textrm{projection spaces}  \\ \textrm{$\FlatRR_{\IndexSubset\opt}$}}}
   }
    \eqfinv
  \end{equation}
%
  and the exposed faces of~$\SupportDualNorm{\TripleNormBall}{k}$ are related to the exposed faces
  of~$\TripleNormBall$ by~\eqref{eq:Face_description_generalized_k-support_norm}. 
\end{proposition}

To the difference
of~\eqref{eq:Support_identification_generalized_k-support_norm}, the left-hand
side
of~\eqref{eq:Support_identification_generalized_k-support_norm_orthant-monotonic}
is exactly the intersection of the level set~${\LevelSet{\lzero}{k}}$ of the
\lzeropseudonorm\ with the exposed
face~$\ExposedFace\bp{\SupportDualNorm{\TripleNormBall}{k},\dual}$, whereas it
was the intersection of a \emph{subset} of the level
set~${\LevelSet{\lzero}{k}}$ of the \lzeropseudonorm\ with the exposed
face~$\ExposedFace\bp{\SupportDualNorm{\TripleNormBall}{k},\dual}$ in the
left-hand side of~\eqref{eq:Support_identification_generalized_k-support_norm}.
\bigskip

\begin{proof}
  The assumptions of
  Proposition~\ref{pr:Support_identification_generalized_k-support_norm} are
  satisfied. Thus, the
  equality~\eqref{eq:Support_identification_generalized_k-support_norm} holds
  true.  The right-hand sides
  of~\eqref{eq:Support_identification_generalized_k-support_norm} and
  of~\eqref{eq:Support_identification_generalized_k-support_norm_orthant-monotonic}
  are identical.  By comparing the left-hand side
  of~\eqref{eq:Support_identification_generalized_k-support_norm} --- namely,
  \( \bigcup_{ \cardinal{\IndexSubset} \leq k} \ProjectionIndexSubset\np{ \TripleNormBall
  } \cap \ExposedFace(\SupportDualNorm{\TripleNormBall}{k},\dual) \) --- with the
  left-hand side
  of~\eqref{eq:Support_identification_generalized_k-support_norm_orthant-monotonic}
  --- namely,
  \( {\LevelSet{\lzero}{k}} \cap
  \ExposedFace(\SupportDualNorm{\TripleNormBall}{k},\dual) \) --- we conclude that
  proving~\eqref{eq:Support_identification_generalized_k-support_norm_orthant-monotonic}
  amounts to showing that
  \begin{equation*}
    \bigcup_{ \cardinal{\IndexSubset} \leq k} 
    \ProjectionIndexSubset\np{ \TripleNormBall } 
    \cap \ExposedFace(\SupportDualNorm{\TripleNormBall}{k},\dual)
    =
    {\LevelSet{\lzero}{k}} \cap
    \ExposedFace(\SupportDualNorm{\TripleNormBall}{k},\dual)
    \eqfinp
  \end{equation*}
  We prove the equality by two opposite inclusions.

  On the one hand, we have that
  \begin{align*}
    \bigcup_{ \cardinal{\IndexSubset} \leq k} \ProjectionIndexSubset\np{\TripleNormBall}
    \cap 
    \ExposedFace(\SupportDualNorm{\TripleNormBall}{k},\dual)
    & \subset
      \bigcup_{ \cardinal{\IndexSubset} \leq k} 
      \FlatRR_{\IndexSubset} \cap
      \ExposedFace(\SupportDualNorm{\TripleNormBall}{k},\dual)
      \tag{as \( \ProjectionIndexSubset\np{\TripleNormBall} \subset \FlatRR_{\IndexSubset}
      \) for all~$\IndexSubset$, by definition of the 
      {orthogonal projection mapping}~$\ProjectionIndexSubset$}
    \\ 
    &=
      {\LevelSet{\lzero}{k}} \cap
      \ExposedFace(\SupportDualNorm{\TripleNormBall}{k},\dual)
      \tag{by~\eqref{eq:LevelSet_lzero_FlatRR_IndexSubset}}
      \eqfinp 
  \end{align*}
  On the other hand, we prove the reverse inclusion
  \( {\LevelSet{\lzero}{k}} \cap
  \ExposedFace(\SupportDualNorm{\TripleNormBall}{k},\dual) 
  \subset \bigcup_{ \cardinal{\IndexSubset} \leq k} \ProjectionIndexSubset\np{\TripleNormBall}
  \cap 
  \ExposedFace(\SupportDualNorm{\TripleNormBall}{k},\dual) \).
  Indeed, for any nonzero dual vector~$\dual \in \DUAL\setminus\na{0}$, we have that
  \begin{align*}
    {\LevelSet{\lzero}{k}} \cap 
    \ExposedFace(\SupportDualNorm{\TripleNormBall}{k},\dual)
    &=
      {\LevelSet{\lzero}{k}} \cap \SupportDualNorm{\TripleNormSphere}{k} \cap
      \ExposedFace(\SupportDualNorm{\TripleNormBall}{k},\dual)
      \tag{because \( \ExposedFace(\SupportDualNorm{\TripleNormBall}{k},\dual)
      \subset \SupportDualNorm{\TripleNormSphere}{k} \) as \( \dual \neq 0 \)}
    \\
    &=
      {\LevelSet{\lzero}{k}} \cap \SupportDualNorm{\TripleNormSphere}{\spacedim} \cap
      \ExposedFace(\SupportDualNorm{\TripleNormBall}{k},\dual)
      \intertext{by \cite[Equation~(37) in
      Proposition~20]{Chancelier-DeLara:2022_OSM_JCA}, giving
      \( {\LevelSet{\lzero}{k}} \cap \SupportDualNorm{\TripleNormSphere}{k} =
      {\LevelSet{\lzero}{k}} \cap \SupportDualNorm{\TripleNormSphere}{\spacedim} \) 
      using the property that $\TripleNorm{\cdot}$ is orthant-monotonic,}
       &=
         {\LevelSet{\lzero}{k}} \cap \TripleNormSphere \cap 
         \ExposedFace(\SupportDualNorm{\TripleNormBall}{k},\dual)
         \intertext{by \cite[Item~2 in Proposition~13]{Chancelier-DeLara:2022_OSM_JCA},
         giving \( \SupportDualNorm{\TripleNorm{\cdot}}{\spacedim} =\TripleNorm{\cdot} \)
         using the property that $\TripleNorm{\cdot}$ is orthant-monotonic,}
       &=
         \bigcup_{ \cardinal{\IndexSubset} \leq k} \FlatRR_{\IndexSubset} \cap \TripleNormSphere \cap 
         \ExposedFace(\SupportDualNorm{\TripleNormBall}{k},\dual)
         \tag{as \( {\LevelSet{\lzero}{k}} =\bigcup_{ \cardinal{\IndexSubset}
         \leq k} \FlatRR_{\IndexSubset} \) by~\eqref{eq:LevelSet_lzero_FlatRR_IndexSubset}}
  \\
  &\subset
    \bigcup_{ \cardinal{\IndexSubset} \leq k} \ProjectionIndexSubset\np{\TripleNormSphere}
    \cap 
    \ExposedFace(\SupportDualNorm{\TripleNormBall}{k},\dual)
    \tag{as \( \FlatRR_{\IndexSubset} \cap \TripleNormSphere
    \subset \ProjectionIndexSubset\np{\TripleNormSphere} \), by definition of the 
    {orthogonal projection mapping}~$\ProjectionIndexSubset$}
   \\
    &\subset
    \bigcup_{ \cardinal{\IndexSubset} \leq k} \ProjectionIndexSubset\np{\TripleNormBall}
    \cap 
    \ExposedFace(\SupportDualNorm{\TripleNormBall}{k},\dual)
    \tag{as \( \ProjectionIndexSubset\np{\TripleNormSphere}
    \subset \ProjectionIndexSubset\np{\TripleNormBall} \) }
    \eqfinp
  \end{align*}
  This ends the proof.
\end{proof}

If, in
Proposition~\ref{pr:Support_identification_generalized_k-support_norm_orthant-monotonic},
we suppose that the source
  norm~$\TripleNorm{\cdot}$ 
  is not only orthant-monotonic, but is monotonic (which is stronger) and symmetric, we get
  the following additional result: to the right-hand side
of~\eqref{eq:Support_identification_generalized_k-support_norm_orthant-monotonic},
\( \argmax_{\cardinal{\IndexSubset}\leq k}
       \TripleNormDual{\ProjectionIndexSubset\dual} \) can be obtained by
       ordering the components of the vector \( \module{\dual} \) and then
       taking indices of the~$k$ largest entries
              (see the comment following the proof of
              Theorem~\ref{th:support_identification}).

\subsection{The case of orthant-strictly monotonic source norms}
\label{The_orthant-strictly_monotonic_case}

The notion of \emph{orthant-strictly monotonic norm} has been introduced
in~\cite[Definition~5]{Chancelier-DeLara:2022_OSM_JCA}: a
norm~$\TripleNorm{\cdot}$ is orthant-strictly monotonic if and only if, for all
$\primal$, $\primal'$ in~$ \RR^{\spacedim}$, we have \np{
  \( |\primal| < |\primal'| \text{ and } \primal~\circ~\primal' \ge 0 \Rightarrow
  \TripleNorm{\primal} < \TripleNorm{\primal'} \) },
where, for any \( \primal=\np{\primal_1,\ldots,\primal_d} \in \RR^{\spacedim} \), we
  denote   \( \module{\primal}
  =\np{\module{\primal_1},\ldots,\module{\primal_d}} \in \RR^{\spacedim} \).
An equivalent characterization is provided in~\cite[Item~3 in
Proposition~6]{Chancelier-DeLara:2022_OSM_JCA}.  A norm~$\TripleNorm{\cdot}$ is
orthant-strictly monotonic if and only if it is \emph{strictly increasing with
  the coordinate subspaces}, in the sense that\footnote{%
  By \( J \subsetneq K \), we mean that \( J \subset K \) and \( J \neq K \). }, for any
\( \primal \in \RR^{\spacedim} \) and any
\( J \subsetneq K \subset\ic{1,{\spacedim}} \), we have
$ \primal_J \neq \primal_K \Rightarrow \TripleNorm{\primal_{J}} < \TripleNorm{\primal_K}$.
An orthant-strictly monotonic norm is orthant-monotonic.

\begin{proposition}
  \label{pr:Support_identification_generalized_k-support_norm_orthant-strictly_monotonic}
  Let $\TripleNorm{\cdot}$ be a (source) norm on~$ \RR^{\spacedim}$, with unit
  ball~$\TripleNormBall$.  For given {sparsity
    threshold~\( k \in \ic{1,{\spacedim}} \)}, we consider the {generalized
    top-$k$ dual~norm}~\( \SupportDualNorm{\DoubleNorm{\cdot}}{{k}} \) (see
  Definition~\ref{de:top_norm}).
  If the source
  norm~$\TripleNorm{\cdot}$ 
  is orthant-strictly
  monotonic, 
  then for 
  any nonzero dual vector~$\dual \in \DUAL\setminus\na{0}$, we have that
  \begin{equation}
        \label{eq:Support_identification_generalized_k-support_norm_orthant-strictly_monotonic}
    \overbrace{
    {\LevelSet{\lzero}{k}} 
    \cap \ExposedFace\bp{ 
        \SupportDualNorm{\TripleNormBall}{k}
        ,
        \underbrace{\dual}%
        _{\substack{\textrm{dual} \\ \textrm{vector}}}
      }
    }^{{\substack{\textrm{$k$-sparse vectors of the exposed face} \\ \textrm{of the unit
        ball~$\SupportDualNorm{\TripleNormBall}{k}$ of} \\ \textrm{the
        generalized $k$-support dual~norm} }}}
    = 
    \Bsetco{ 
      \overbrace{ 
        \ExposedFace\np{ 
          \TripleNormBall
          ,
        \underbrace{\ProjectionIndexSubsetOpt\dual}%
            _{\substack{\textrm{$k$-sparse}\\ \textrm{projection} \\
          \textrm{of the} \\ \textrm{dual vector} \\ \textrm{on
            $\FlatRR_{\IndexSubset\opt}$} } }
    }
           }^{{\substack{\textrm{exposed face of} \\ \textrm{the unit
        ball~$\TripleNormBall$} \\ \textrm{of the source norm} }}}
    }
    {\underbrace{\IndexSubset\opt \in
       \argmax_{\cardinal{\IndexSubset}\leq k}
       \TripleNormDual{\ProjectionIndexSubset\dual}}%
         _{\substack{\textrm{selection of} \\ \textrm{$k$-sparse}\\
        \textrm{projection spaces}  \\ \textrm{$\FlatRR_{\IndexSubset\opt}$}}}
   }
    \eqfinv
  \end{equation}
%
  and the exposed faces of~$\SupportDualNorm{\TripleNormBall}{k}$ are related to the exposed faces
  of~$\TripleNormBall$ by
  \begin{equation}
    \ExposedFace(\SupportDualNorm{\TripleNormBall}{k},\dual)=
    \closedconvexhull
    \Bsetco{
      \ExposedFace(\TripleNormBall,\ProjectionIndexSubsetOpt\dual)}
    {\IndexSubset\opt \in 
      \argmax_{\cardinal{\IndexSubset}\leq k} \TripleNormDual{\ProjectionIndexSubset\dual}}
    \eqfinp
    \label{eq:Face_description_generalized_k-support_norm_orthant-strictly_monotonic}
  \end{equation}
\end{proposition}
To the difference of the right-hand sides
of~\eqref{eq:Support_identification_generalized_k-support_norm},
\eqref{eq:Face_description_generalized_k-support_norm}, and
\eqref{eq:Support_identification_generalized_k-support_norm_orthant-monotonic},
there is no projection
\( \ProjectionIndexSubsetOpt\bp{
  \ExposedFace(\TripleNormBall,\ProjectionIndexSubsetOpt\dual)} \), but just
\( \ExposedFace(\TripleNormBall,\ProjectionIndexSubsetOpt\dual) \) in the
right-hand sides
of~\eqref{eq:Support_identification_generalized_k-support_norm_orthant-strictly_monotonic}
and
\eqref{eq:Face_description_generalized_k-support_norm_orthant-strictly_monotonic}.
\bigskip

\noindent
\begin{proof}
  An orthant-strictly monotonic norm is orthant-monotonic, the assumptions of
  Proposition~\ref{pr:Support_identification_generalized_k-support_norm_orthant-monotonic}
  hold true.  Thus,
  Equation~\eqref{eq:Support_identification_generalized_k-support_norm_orthant-monotonic}
  holds true and hence, to
  prove~\eqref{eq:Support_identification_generalized_k-support_norm_orthant-strictly_monotonic},
  it suffices to show that
  \( \pi_{\IndexSubset\opt}\bp{
    \ExposedFace(\TripleNormBall,\pi_{\IndexSubset\opt}\dual)} =
  \ExposedFace(\TripleNormBall,\pi_{\IndexSubset\opt}\dual) \).\\\\
  First, notice that, for any \( \IndexSubset\opt \in 
  \argmax_{\cardinal{\IndexSubset}\leq k}
  \TripleNormDual{\ProjectionIndexSubset\dual} \), we have that 
  \( \TripleNormDual{\pi_{\IndexSubset\opt}\dual} > 0 \).
  Indeed, on the contrary, we would have that
  \( \max_{\cardinal{\IndexSubset}\leq k}
  \TripleNormDual{\ProjectionIndexSubset\dual} =0\), hence that
  \( \ProjectionIndexSubset\dual=0 \) for any $\IndexSubset$ with
  $\cardinal{\IndexSubset}\leq k$. As $k\geq 1$, this would imply that
  $\dual=0$.\\\\
  Second, we have that
  \begin{align*}
    \primal\in \ExposedFace(\TripleNormBall,\pi_{\IndexSubset\opt}\dual)
    &\iff
      \proscal{\primal}{\pi_{\IndexSubset\opt}\dual} =
      \max_{\primalbis\in\TripleNormBall} \proscal{\primalbis}{\pi_{\IndexSubset\opt}\dual}
            \mtext{ and } \primal\in\TripleNormBall
      \tag{by definition~\eqref{eq:ExposedFace} of the exposed face
      \( \ExposedFace(\TripleNormBall,\pi_{\IndexSubset\opt}\dual) \)} 
    \\
    &\iff
      \proscal{\primal}{\pi_{\IndexSubset\opt}\dual} =
      \TripleNormDual{\pi_{\IndexSubset\opt}\dual} 
            \mtext{ and } \primal\in\TripleNormBall
      \tag{by definition of the dual norm~$\TripleNormDual{\cdot}$}
    \\
   &\iff
      \proscal{\primal}{\pi_{\IndexSubset\opt}\dual} =
      \TripleNorm{\primal} \TripleNormDual{\pi_{\IndexSubset\opt}\dual}
      \mtext{ and } \TripleNorm{\primal}=1
    \\
    &\iff
      \proscal{\pi_{\IndexSubset\opt}\primal}{\pi_{\IndexSubset\opt}\dual} =
      \TripleNorm{\primal} \TripleNormDual{\pi_{\IndexSubset\opt}\dual}
      \mtext{ and } \TripleNorm{\primal}=1
      \tag{as the orthogonal projection~\( \ProjectionIndexSubset \) is self-adjoint
  (see~\eqref{eq:orthogonal_projection_self-dual})}
    \\
    &\implies
      \TripleNorm{\pi_{\IndexSubset\opt}\primal} \TripleNormDual{\pi_{\IndexSubset\opt}\dual}
      \geq \proscal{\pi_{\IndexSubset\opt}\primal}{\pi_{\IndexSubset\opt}\dual} =
      \TripleNorm{\primal} \TripleNormDual{\pi_{\IndexSubset\opt}\dual}
      \tag{by polar inequality}
    \\
    &\implies
      \TripleNorm{\pi_{\IndexSubset\opt}\primal}\geq  \TripleNorm{\primal}
      \tag{since             \( \TripleNormDual{\pi_{\IndexSubset\opt}\dual} > 0 \)}
    \\
    &\implies
      \TripleNorm{\pi_{\IndexSubset\opt}\primal} = \TripleNorm{\primal}
      \intertext{because the norm~$\TripleNorm{\cdot}$ is orthant-strictly
      monotonic, hence  orthant-monotonic,
      hence \( \TripleNorm{\pi_{\IndexSubset\opt}\primal}\leq
      \TripleNorm{\primal} \)}
    &\implies
      \pi_{\IndexSubset\opt}\primal=\primal
      \eqfinv
  \end{align*}
  because if we had \( \pi_{\IndexSubset\opt}\primal\neq\primal \), we would conclude
  that \( \TripleNorm{\pi_{\IndexSubset\opt}\primal} < \TripleNorm{\primal} \),
as the norm~$\TripleNorm{\cdot}$ is orthant-strictly monotonic. 

  We conclude that 
  \( \pi_{\IndexSubset\opt}\bp{
    \ExposedFace(\TripleNormBall,\pi_{\IndexSubset\opt}\dual)}
  =  \ExposedFace(\TripleNormBall,\pi_{\IndexSubset\opt}\dual) \).
\end{proof}

\section{Geometry of the top-($q$,$k$) and \lpsupportnorm{p}{k}s}
\label{Geometry_of_lptopnorm{q}{k}_and__lpsupportnorm{p}{k}s_unit_balls}

We recall that ${\spacedim} \in \NN^*$ is a natural number, and that 
we consider the finite-dimensional real Euclidean vector space~$\RR^{\spacedim}$
equipped with the scalar product~\( \proscal{\,}{} \).

This section is devoted to the geometric analysis of the faces and normal
cones\footnote{%
  Recall that normal cones are in one-to-one correspondence with exposed faces;
  each can be recovered from the other by duality. When the unit ball is
  polytopal, as in~\S\ref{The_case_where_p_is_equal_to_1_or_to_+infty}, every
  face is exposed.}
of the unit balls of the generalized top-$k$ and $k$-support dual~norms,
as defined in Sect.~\ref{The_case_of_generalized_top-k_and_k-support_norms},
when they are generated by the $\ell_p$ source norms 
    \( \TripleNorm{\cdot} = \Norm{\cdot}_{p} \) for $p\in [1,\infty]$.
    In~\S\ref{lptopnorm{q}{k}_and__lpsupportnorm{p}{k}s}, we recall the definition
    of the $\ell_p$-norms, and then of \lpsupportnorm{p}{k}s and
    \lptopnorm{q}{k}s where $1/p + 1/q=1$.
Then, we relate these latter to the generalized top-$k$ and $k$-support dual~norms,
as defined in Sect.~\ref{The_case_of_generalized_top-k_and_k-support_norms},
when the source norm is the $\ell_p$ norm.
%
In~\S\ref{The_case_where_p_is_equal_to_1_or_to_+infty}, we study the case where
$p$ is equal to $+\infty$; we recall known results, where the unit
balls are polytopal and therefore all faces are exposed.
In~\S\ref{The_case_where_1<p<+infty}, we study the case where $1<p<+\infty$.
Here, we present new results. Rather strikingly, we show that every proper
face, exposed or not, of the unit ball of the $k$-support dual~norms is a
hypersimplex, i.e, the convex hull of $0/1$-valued points with the same
$\lzero$-norm (see~Equation~\eqref{eq:lzeropseudonorm}).
This property connects geometry of unit balls with the more analytical
  original motivation, as exposed in~\S\ref{Motivation}, of looking for
  sparsity-inducing norms.

\subsection{The \lptopnorm{q}{k} and  \lpsupportnorm{p}{k}s}
\label{lptopnorm{q}{k}_and__lpsupportnorm{p}{k}s}

For any $p$ in $[1,+\infty[$ and $\primal$ in $\RR^{\spacedim}$, let us recall that
the $\ell_p$-norm of $\primal$ is
\[
  \norm{\primal}_{p} = \Bp{\sum_{i=1}^{\spacedim} \abs{\primal_i}^p}^{\frac{1}{p}}
  \eqfinv 
\]
and that its $\ell_\infty$-norm is
\[
  \norm{\primal}_{\infty} = \max_{i\in\Vset} |\primal_i| \eqfinp
\]

For any $p$ in $[1,+\infty]$, we denote by~\( {\Ball}_{p} \)
and~\( {\Sphere}_{p} \) the unit ball and the unit sphere for the
$\ell_p$-norm. When the source norm is the $\ell_p$-norm, 
following Definition~\ref{de:top_norm}, 
we get that\footnote{%
We refer the reader to \cite[\S3.1]{Chancelier-DeLara:2022_OSM_JCA}, but
  where there the source norm is~\(\norm{\cdot}_{p}\), whereas here the 
source norm is the dual norm \(\norm{\cdot}_{q}\) of  \(\norm{\cdot}_{p}\), with \( 1/p + 1/q =1 \).
\label{ft:lpnorms}
}
\begin{itemize}
\item[$\circ$]
  the corresponding generalized $k$-support dual~norm 
  \( \SupportDualNorm{ \bp{ \norm{\cdot}_{p} } }{k}  \)
  is the \emph{\lpsupportnorm{p}{k}} denoted by~\( \LpSupportNorm{\cdot}{p}{k} \),
  with unit ball~\( \LpSupportBall{\Ball}{p}{k} \)
  and unit sphere~\( \LpSupportBall{\Sphere}{p}{k} \),
\item[$\circ$]
  the corresponding generalized top-$k$ dual~norm
  \( \TopDualNorm{ \bp{ \norm{\cdot}_{p} } }{k} \)
  is the \emph{\lptopnorm{q}{k}} denoted by~\( \LpTopNorm{\cdot}{q}{k} \), where \( 1/p + 1/q =1 \), 
  with unit ball~\( \LpTopBall{\Ball}{q}{k} \)
  and unit sphere~\( \LpTopBall{\Sphere}{q}{k} \).
\end{itemize}
For any $p$ and $q$ in $[1,+\infty]$ such that \( 1/p + 1/q =1 \), we have
\begin{equation}
  \LpTopNorm{\cdot}{p}{k} = \SupportFunction{\LpSupportBall{\Ball}{q}{k}}
  \eqsepv
  \LpTopBall{\Ball}{p}{k} = \PolarSet{\bp{\LpSupportBall{\Ball}{q}{k}}}
  \mtext{ and }
  \LpSupportNorm{\cdot}{q}{k} = \SupportFunction{\LpTopBall{\Ball}{p}{k}}
  \eqsepv
  \LpSupportBall{\Ball}{q}{k} = \PolarSet{\bp{\LpTopBall{\Ball}{p}{k}}}
  \eqfinp
\end{equation}

The norms obtained when $p$ varies from $1$ to $+\infty$ are summarized in
Table~\ref{tab:Examples_of_generalized_top-k_and_k-support_dual_norms}. The
top-($1$,$k$) and top-($2$,$k$) norms arise in various contexts under different
names (see~\cite{Gotoh-Takeda-Tono:2018} and references therein). They are called
the \emph{vector $k$-norm} in~\cite[Sect.~2]{Wu-Ding-Sun-Toh:2014}, the
\emph{largest $k$-norm} or \emph{CVaR norm} for the $\ell_\infty$-norm
in~\cite[Sect.~1]{Gotoh-Uryasev:2016}, the \emph{$2$-$k$-symmetric gauge norm}
in~\cite{Mirsky:1960}, and the \emph{Ky Fan vector norm} for the $\ell_2$-norm
in~\cite{Obozinski-Bach:hal-01412385}.
Similarly, the \lpsupportnorm{2}{k} is referred to as {$k$-support norm}
in~\cite{Argyriou-Foygel-Srebro:2012}.
The {\lpsupportnorm{p}{k}} for $p\in [1,\infty]$ is defined in
\cite[Definition~21]{McDonald-Pontil-Stamos:2016} where it is showed that the
dual norm of the \lptopnorm{p}{k} is the \lpsupportnorm{q}{k}, where \( 1/p + 1/q = 1 \).
Therefore (see Footnote~\ref{ft:lpnorms}), the generalized $k$-support dual~norm is the \lpsupportnorm{p}{k}
(denoted by~\( \LpSupportNorm{\cdot}{p}{k} \)) when the source
norm~$\TripleNorm{\cdot}$ is the $\ell_p$-norm $\Norm{\cdot}_{p}$.

Let us briefly discuss the cases when $p$ is equal to $1$ or to $+\infty$.  When
$p$ is equal to $1$, it follows from the definition that the unit ball
$\LpSupportBall{\Ball}{p}{k}$ is the cross-polytope~$B_1$, independently of~$k$,
and that its polar $\LpTopBall{\Ball}{q}{k}$ coincides with the unit hypercube
$B_\infty$. When $p$ is equal to $+\infty$, the balls
$\LpSupportBall{\Ball}{p}{k}$ and $\LpTopBall{\Ball}{q}{k}$ form two families of
polytopes that interpolate between the cross-polytope and the
hypercube~\cite{DezaHiriart-UrrutyPournin2021}, as illustrated in
Figures~\ref{fig2} and~\ref{fig1} when the dimension~$\spacedim$ is equal to~$3$.
\medskip

In \S\ref{The_case_where_p_is_equal_to_1_or_to_+infty}, we will apply the result of
Proposition~\ref{pr:Support_identification_generalized_k-support_norm_orthant-monotonic}
to the orthant-monotonic norm~\( \ell_{\infty}\), and obtain a characterization of
\( {\LevelSet{\lzero}{k}} \cap \ExposedFace(\LpSupportBall{\Ball}{\infty}{k},\dual) \)
in terms of the sets
\( \projection_{\IndexSubset}\bp{{\Sphere}_{\infty}\cap
  \ExposedFace({\Ball}_{\infty},\projection_{\IndexSubset}\dual)} \) for certain
subsets $\IndexSubset$ of $\Vset$.
In~\S\ref{The_case_where_1<p<+infty}, we will apply the result of
Proposition~\ref{pr:Support_identification_generalized_k-support_norm_orthant-strictly_monotonic}
to the orthant-strictly monotonic norms~\( \ell_p\), where $p$ belongs to
$[1,\infty[$, and obtain a characterization of
\( {\LevelSet{\lzero}{k}} \cap \ExposedFace(\LpSupportBall{\Ball}{p}{k},\dual) \)
in terms of the sets
\( {\Sphere}_{p}\cap \ExposedFace({\Ball}_{p},\projection_{\IndexSubset}\dual) \)
for certain subsets $\IndexSubset$ of $\Vset$.

\begin{table}
  \centering
  \begin{tabular}{||c||c|c||}
    \hline\hline 
    {source norm} 
    & generalized top-$k$ dual~norms
    & generalized $k$-support dual~norms
    \\
    \( \TripleNorm{\cdot} \) 
    & \( \TopDualNorm{\TripleNorm{\cdot}}{k} \), $k\in\Vset$
    & \( \SupportDualNorm{\TripleNorm{\cdot}}{k} \), $k\in\Vset$
    \\
    \hline\hline 
    \( \Norm{\cdot}_{p} \)
    & \lptopnorm{q}{k} 
    & \lpsupportnorm{p}{k} 
    \\
    & \( \LpTopNorm{\dual}{q}{k} \) %
    & \( \LpSupportNorm{\primal}{p}{k} \) %
    \\
    & \(\LpTopNorm{\dual}{q}{k}=\bp{ \sum_{l=1}^{k} \module{ \dual_{\nu(l)} }^q }^{\frac{1}{q}} \)
    & no analytic expression %
   \\[1mm]
    & &
    \\
    & \( \LpTopNorm{\dual}{q}{1}= \Norm{\dual}_{\infty} \)
    & \( \LpSupportNorm{\primal}{p}{1} = \Norm{\primal}_{1} \)
    \\
    \hline
    \( \Norm{\cdot}_{1} \) 
    & \lptopnorm{\infty}{k} 
    & \lpsupportnorm{1}{k} 
    \\
    & $\ell_{\infty}$-norm 
    & $\ell_{1}$-norm
    \\[1mm]
    & &
    \\
    & \( \LpTopNorm{\dual}{\infty}{k} 
      = \Norm{\dual}_{\infty} \), $\forall k\in\Vset$
    & \( \LpSupportNorm{\primal}{1}{k} = \Norm{\primal}_{1} \), $\forall k\in\Vset$
    \\
    \hline 
    \( \Norm{\cdot}_{2} \) 
    & \lptopnorm{2}{k} 
    & \lpsupportnorm{2}{k} 
    \\
    & \( \LpTopNorm{\dual}{2}{k} = \sqrt{ \sum_{l=1}^{k} \module{ \dual_{\nu(l)} }^2 } \)
    & \( \LpSupportNorm{\primal}{2}{k} \)
      no analytic expression
    \\
    & & (computation \cite[Prop. 2.1]{Argyriou-Foygel-Srebro:2012})
    \\[1mm]
    & &
    \\
    & \( \LpTopNorm{\dual}{2}{1}= \Norm{\dual}_{\infty} \)
    & \( \LpSupportNorm{\primal}{2}{1} = \Norm{\primal}_{1} \)
    \\
    \hline 
    \( \Norm{\cdot}_{\infty} \)
    & \lptopnorm{1}{k} 
    & \lpsupportnorm{\infty}{k} 
    \\
    & \( \LpTopNorm{\dual}{1}{k} = \sum_{l=1}^{k} \module{ \dual_{\nu(l)} } \)
    & \( \LpSupportNorm{\primal}{\infty}{k} =
      \max \na{ \frac{\Norm{\primal}_{1}}{k} , \Norm{\primal}_{\infty} } \) 
    \\[1mm]
    & &
    \\
    & \( \LpTopNorm{\dual}{1}{1}= \Norm{\dual}_{\infty} \)
    & \( \LpSupportNorm{\primal}{\infty}{1} = \Norm{\primal}_{1} \)

    \\
    \hline\hline
  \end{tabular}
  \caption{Examples of generalized top-$k$ and $k$-support dual~norms
    generated by the $\ell_p$ source norms 
    \( \TripleNorm{\cdot} = \Norm{\cdot}_{p} \) for $p\in [1,\infty]$ and $1/p +
    1/q =1$
     (from \cite[Table~1]{Chancelier-DeLara:2022_OSM_JCA}, but see Footnote~\ref{ft:lpnorms}).
    For \( \dual \in \RR^{\spacedim} \), $\nu$ denotes a permutation of \( \Vset \) such that
    \( \module{ \dual_{\nu(1)} } \geq \module{ \dual_{\nu(2)} } 
    \geq \cdots \geq \module{ \dual_{\nu(\spacedim)} } \).
  }
    \label{tab:Examples_of_generalized_top-k_and_k-support_dual_norms}
\end{table}

\subsection{The case when $p$ is equal to $+\infty$}
\label{The_case_where_p_is_equal_to_1_or_to_+infty}

When the source norm is the $\ell_\infty$-norm, the corresponding last row of
Table~\ref{tab:Examples_of_generalized_top-k_and_k-support_dual_norms} tells us
that we should study the unit balls of \lptopnorm{1}{k}s and
\lpsupportnorm{\infty}{k}s.  This results in families of polytopes whose geometry and
combinatorics have been studied in~\cite{DezaHiriart-UrrutyPournin2021}. In this section,
we review these families of polytopes.  Following the notation of
Coxeter, denote by $\gamma_{\spacedim}$ the ${\spacedim}$-dimensional hypercube
$[-1,1]^{\spacedim}$ and by $\beta_{\spacedim}$ the cross-polytope whose vertices
are the centers of the facets of $\gamma_{\spacedim}$.  Note that these two polytopes
are related by polarity. It follows from \cite[Equations (1.2) and
(1.3)]{DezaHiriart-UrrutyPournin2021} that, for every $k$ in $\llbracket1,{\spacedim}\rrbracket$,
\begin{equation}
  \LpTopBall{\Ball}{1}{k}=\conv\biggl(\beta_{\spacedim}\cup\frac{1}{k}\gamma_{\spacedim}\biggr)
  \text{ and } 
  \LpSupportBall{\Ball}{\infty}{k}=k\beta_{\spacedim}\cap\gamma_{\spacedim}
  \eqfinp
\end{equation}

\begin{figure}
  \begin{center}
          \mbox{\includegraphics[scale=1.3]{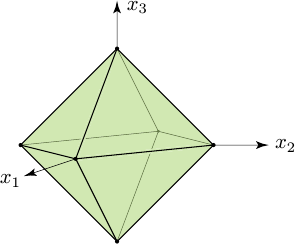}} \qquad
          \mbox{\includegraphics[scale=1.3]{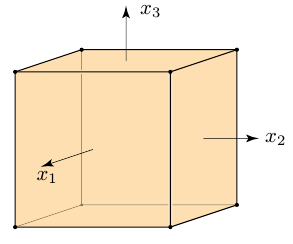}}    
  \end{center}
  \caption{Unit balls $\LpSupportBall{\Ball}{\infty}{1}$ (left) and
      $\LpTopBall{\Ball}{1}{1}$ (right) when $\spacedim=3$}
    \label{fig2}
\end{figure}

\begin{figure}
  \begin{center}
          \mbox{\includegraphics[scale=1.3]{DualL1alt.pdf}} \qquad
          \mbox{\includegraphics[scale=1.3]{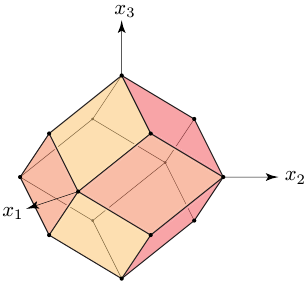}}    
  \end{center}
  \caption{Unit balls $\LpSupportBall{\Ball}{\infty}{2}$ (left) and
      $\LpTopBall{\Ball}{1}{2}$ (right) when $\spacedim=3$}
    \label{fig1}
\end{figure}

When $k$ is equal to $1$, $\LpTopBall{\Ball}{1}{k}$ coincides with the hypercube
$\gamma_{\spacedim}$ and $\LpSupportBall{\Ball}{\infty}{k}$ with the cross-polytope
$\beta_{\spacedim}$. When $k$ is equal to ${\spacedim}$, the opposite holds:
$\LpTopBall{\Ball}{1}{k}$ coincides with the cross-polytope $\beta_{\spacedim}$ and
$\LpSupportBall{\Ball}{\infty}{k}$ with the hypercube $\gamma_{\spacedim}$. In particular
these two families interpolate between the hypercube and the cross-polytope and,
as pointed out in \cite{DezaHiriart-UrrutyPournin2021},
$\LpTopBall{\Ball}{1}{k}$ and $\LpSupportBall{\Ball}{\infty}{k}$ are related by
polarity for all $k$ in $\llbracket1,{\spacedim}\rrbracket$ and not just when
$k$ is equal to $1$ or ${\spacedim}$.  Note that
in~\cite{DezaHiriart-UrrutyPournin2021} the parameter $k$ is allowed to take any
(possibly non integral) value in the interval $[1,{\spacedim}]$. In dimension
$3$, these two polytopes are shown on Figure~\ref{fig2} when $k$ is equal to $1$
and in Figure~\ref{fig1} when $k$ is equal to $2$. Theorem 2.1 from
\cite{DezaHiriart-UrrutyPournin2021} can be rephrased as follows.

\begin{theorem}
  \label{CDLDP.sec.3.2.thm.1}
  The facets of $\LpTopBall{\Ball}{1}{k}$ are precisely the sets of the form
  \begin{equation}\label{CDLDP.sec.3.2.thm.1.eq.1}
    \conv\biggl(\ExposedFace(\beta_{\spacedim},\dual)\cup{\frac{1}{k}\ExposedFace(\gamma_{\spacedim},\dual)}\biggr)
  \end{equation}
  where $\dual$ is a vector in $\{-1,0,1\}^{\spacedim}$ with exactly $k$ nonzero coordinates.
\end{theorem}

Observe that (\ref{CDLDP.sec.3.2.thm.1.eq.1}) is precisely the exposed face
$\ExposedFace(\LpTopBall{\Ball}{1}{k},\dual)$. As noted
in~\cite{DezaHiriart-UrrutyPournin2021}, for any vector $\dual$ in
$\{-1,0,1\}^{\spacedim}$, the affine hulls of
$\ExposedFace(\beta_{\spacedim},\dual)$ and
$\ExposedFace(\gamma_{\spacedim},\dual)$ are orthogonal subspaces of
$\mathbb{R}^{\spacedim}$. More precisely, if we denote by $k$ the number of nonzero
coordinates of $\dual$, hence $k=\lzero\np{\dual}$, then the two polytopes
$\ExposedFace(\beta_{\spacedim},\dual)$ and
$\ExposedFace(\gamma_{\spacedim},\dual)/k$ intersect in a single point that belongs
to the relative interior of both of these polytopes. We then get, as an
immediate consequence of Theorem \ref{CDLDP.sec.3.2.thm.1}, the following description of all the proper faces of
$\LpTopBall{\Ball}{1}{k}$.

\begin{corollary}\label{CDLDP.sec.3.2.cor.1}
  The proper faces of $\LpTopBall{\Ball}{1}{k}$ are precisely the sets of the form
  \begin{equation}
    \conv\biggl(F\cup{\frac{1}{k}G}\biggr)
  \end{equation}
  where, for some vector $\dual$ in $\{-1,0,1\}^{\spacedim}$ with exactly $k$ nonzero coordinates,
  \begin{enumerate}
  \item[(i)] $F$ and $G$ are exposed faces of $\ExposedFace(\beta_{\spacedim},\dual)$
    and $\ExposedFace(\gamma_{\spacedim},\dual)$, respectively,
  \item[(ii)] $F$ and $G$ are not both empty, 
  \item[(iii)] $F$ is equal to $\ExposedFace(\beta_{\spacedim},\dual)$ if and only if $G$
    is equal to $\ExposedFace(\gamma_{\spacedim},\dual)$.
  \end{enumerate}
\end{corollary}

Corollary \ref{CDLDP.sec.3.2.cor.1} completely describes the faces of
$\LpTopBall{\Ball}{1}{k}$ (which is further enumerated in
\cite{DezaHiriart-UrrutyPournin2021}). Since $\LpSupportBall{\Ball}{\infty}{k}$ is
the polar of $\LpTopBall{\Ball}{1}{k}$, the normal cones of
$\LpSupportBall{\Ball}{\infty}{k}$ are precisely the cones spanned by the exposed faces of
$\LpTopBall{\Ball}{1}{k}$ and, as a consequence, Corollary
\ref{CDLDP.sec.3.2.cor.1} also describes the normal fan of
$\LpSupportBall{\Ball}{\infty}{k}$.
Since $\LpTopBall{\Ball}{1}{k}$ is polytopal, all its proper faces are exposed,
and the same holds by polarity for $\LpSupportBall{\Ball}{\infty}{k}$. Moreover,
because both unit balls are polytopal, faces are in a one-to-one correspondence
with normal cones. Thus, one can recover either from the other.

\subsection{The case when $1<p<+\infty$}
\label{The_case_where_1<p<+infty}

When the source norm is the $\ell_p$-norm where $1<p<+\infty$, the first row of
Table~\ref{tab:Examples_of_generalized_top-k_and_k-support_dual_norms} tells us
that we should study the unit balls of the \lptopnorm{q}{k}, with $1/p+1/q=1$,
and its dual \lpsupportnorm{p}{k}. Thus, we will describe the exposed faces and
the normal cones of $\LpSupportBall{\Ball}{p}{k}$. The exposed faces and normal
cones of $\LpTopBall{\Ball}{p}{k}$ can then be recovered by polarity. The balls
$\LpSupportBall{\Ball}{2}{2}$ and $\LpTopBall{\Ball}{2}{2}$ are shown on
Figure~\ref{CDLDP.sec.3.3.fig.1} when ${\spacedim}$ is equal to $3$. One can see
that $\LpTopBall{\Ball}{2}{2}$ is the intersection of three cylinders colored
yellow, orange, and red. By duality, $\LpSupportBall{\Ball}{2}{2}$ has eight
triangular faces. While these triangular faces are exposed, their edges, shown
as dotted lines, are faces of $\LpSupportBall{\Ball}{2}{2}$ that are not
exposed.

\begin{figure}
  \begin{center}
    \mbox{\includegraphics[scale=1.3]{DualL2alt.pdf}} \qquad
    \mbox{\includegraphics[scale=1.3]{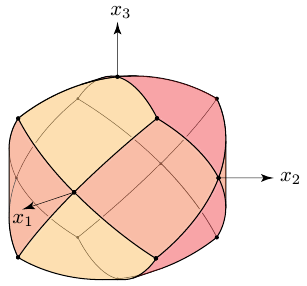}}    
  \end{center}
  \caption{Unit balls $\LpSupportBall{\Ball}{2}{2}$ (left) and
      $\LpTopBall{\Ball}{2}{2}$ (right) when $\spacedim=3$}
  \label{CDLDP.sec.3.3.fig.1}
\end{figure}

The face of $\Ball_{p}$ exposed by a nonzero vector $\dual$ from
$\mathbb{R}^{\spacedim}$ can be recovered from the equality case of H{\"o}lder's
inequality. Indeed, by this inequality,
\begin{equation}\label{eq:Holder}
\bigl|\proscal{z\!}{\!\!y}\bigr|\leq\|\dual\|_q
\end{equation}
for any point~$z$ in $\Ball_{p}$ with equality when
\begin{equation}\label{eq:Holdereq}
  |z_i|^p=\lambda|\dual_i|^q
\end{equation}
for some nonnegative number~$\lambda$ and all integers $i \in \Vset$, where
\begin{equation}
  \frac{1}{p}+\frac{1}{q}=1
  \eqfinp
\end{equation}

Now assume that~$z$ belongs to the face of $\Ball_{p}$ exposed by $\dual$. In that
case, $\|z\|_p$ is equal to~$1$ and (\ref{eq:Holder}) must turn into an
equality. In particular, there exists a nonnegative number~$\lambda$ satisfying
(\ref{eq:Holdereq}) for every $i$. The value of~$\lambda$ can be recovered by summing
(\ref{eq:Holdereq}) over $i$:
\begin{equation}
  \lambda=\frac{1}{\|\dual\|_q^q}
  \eqfinp
\end{equation}
As a consequence, $z$ must be the point such that, when $\dual_i$ is equal to $0$
then so is $z_i$ and when $\dual_i$ is nonzero, then $z_i$ is the number with the
same sign than $\dual_i$ satisfying
\begin{equation}
  |z_i|=\biggl(\frac{|\dual_i|}{\|\dual\|_{q}}\biggr)^{\!q/p}
  \eqfinp
  \label{vp_def}
\end{equation}
Hence, the face of $\Ball_{p}$ exposed by $\dual$ is made of just the point~$z$. From now
on, we shall denote~$z$ by $v_p(\dual)$. Note that $v_2(\dual)$ and $\dual$ are
colinear, more precisely,
\begin{equation}
  v_2(\dual)=\frac{\dual}{\|\dual\|_2}
  \eqfinp
\end{equation}
We are now able to characterize the exposed faces of $\LpSupportBall{\Ball}{p}{k}$.
\begin{theorem}\label{th:faces-of-BTstar2k}
For any number $p$ satisfying $1<p<+\infty$, the face of
  $\LpSupportBall{\Ball}{p}{k}$ exposed by a given nonzero vector~$\dual$ is the
  convex hull of all the points of the form $v_p(\pi_{\IndexSubset\opt}\dual)$ where
  \begin{equation}
    \IndexSubset\opt\in \argmax_{\cardinal{\IndexSubset}\leq k}{\|{\projection_{\IndexSubset}\dual}}\|_{1}
  \end{equation}
is obtained by ordering the components of the vector \( \module{\dual} \) and then
       taking indices of the~$k$ largest entries.
\end{theorem}

\begin{proof}
  Using
  Equation~\eqref{eq:Face_description_generalized_k-support_norm_orthant-strictly_monotonic}
  for the orthant-strictly monotonic $\ell_p$-norm, one obtains that
  \begin{equation*}
    \ExposedFace\Bigl(\LpSupportBall{\Ball}{p}{k},\dual\Bigr)=
    \closedconvexhull
    \Bsetco{ \ExposedFace(\Ball_{p},\pi_{\IndexSubset\opt}\dual) }
    {\IndexSubset\opt \in 
      \argmax_{\cardinal{\IndexSubset}\leq k} \norm{\projection_{\IndexSubset}\dual}_p}
    \eqfinp
  \end{equation*}
  Consider any size $k$ subset $\IndexSubset\opt$ of
  $\llbracket1,{\spacedim}\rrbracket$ that belongs to
  ${\argmax_{\cardinal{\IndexSubset}\leq k}
    \norm{\projection_{\IndexSubset}\dual}_p}$. As seen in
  Equation~\eqref{vp_def}, the face of $\Ball_{p}$ exposed by
  $\pi_{\IndexSubset\opt}\dual$ is the vertex
  $v_p(\pi_{\IndexSubset\opt}\dual)$. Hence, we get that
  \begin{equation*}
    \ExposedFace\Bigl(\LpSupportBall{\Ball}{p}{k},\dual\Bigr)=
    \closedconvexhull
    \Bsetco{ v_p(\pi_{\IndexSubset\opt}\dual)}
    {\IndexSubset\opt \in 
      \argmax_{\cardinal{\IndexSubset}\leq k} \norm{\projection_{\IndexSubset}\dual}_p}
    \eqfinp
  \end{equation*}
  %
As the \( \ell_{p} \)-norm \( \| \cdot \|_{p} \) is monotonic and
    symmetric, the above argmax can be obtained by
       ordering the components of the vector \( \module{\dual} \) and then
       taking indices of the~$k$ largest entries
       (see also the comment following the proof of Theorem~\ref{th:support_identification}). 
As the argument is valid for $p=1$, we get that 
  \begin{equation*}
    \argmax_{\cardinal{\IndexSubset}\leq k} \|\projection_{\IndexSubset}\dual\|_p
    = \argmax_{\cardinal{\IndexSubset}\leq k}{\|{\projection_{\IndexSubset}\dual}}\|_{1}\mbox{,}
  \end{equation*}
  which completes the proof.
\end{proof}

Let us first introduce some notation. Consider an integer $k$ in
$\llbracket1,{\spacedim}\rrbracket$ and a nonzero vector $\dual$ from
$\mathbb{R}^{\spacedim}$.  Denote by $m_k(\dual)$ the largest number such that the set
\begin{equation*}
  \bsetco{i\in\Vset}{|\dual_i|\geq{m_k(\dual)}}
\end{equation*}
contains at least $k$ indices. In other words,
\begin{equation*}
  m_k(\dual) = \sup \Bsetco{\lambda \geq 0}{\bcardinal{\nsetco{i\in \Vset}{|\dual_i|\geq{\lambda}}} \geq k }
  \eqfinp 
\end{equation*}
We will refer by $L_k(\dual)$ to the set of the indices $i$ such that $|\dual_i|$ is
greater than $m_k(\dual)$ and by $\overline{L}_k(\dual)$ the set of the indices
$i$ such that $|\dual_i|$ is greater than or equal to $m_k(\dual)$:
\begin{subequations}
  \begin{align}
    L_k(\dual) &= \defsetco{i \in \ic{1,{\spacedim}} }{|\dual_i| > m_k(\dual)}
                 \eqfinv 
    \\
    \overline{L}_k(\dual) &= \defsetco{i \in \ic{1,{\spacedim}} }{|\dual_i| \geq m_k(\dual)}
                            \eqfinp              
  \end{align}  
\end{subequations}
The following statement is an immediate consequence of these definitions.

\begin{proposition}\label{CDLDP.sec.3.3.prop.1}
  For any integer $k$ in $\llbracket1,{\spacedim}\rrbracket$ and any nonzero
  vector $\dual$ in $\mathbb{R}^{\spacedim}$,
  we have that 
  \[
  L_k(\dual)=\bigcap_{\IndexSubset\opt}\IndexSubset\opt
\eqsepv 
  \overline{L}_k(\dual)=\bigcup_{\IndexSubset\opt}\IndexSubset\opt
\eqfinv
  \]
  where the union and the intersection range over the elements
  $\IndexSubset\opt$ of
  $\argmax_{\cardinal{\IndexSubset}\leq k}
  \norm{\projection_{\IndexSubset}\dual}_1$.
\end{proposition}

Using these notations, Theorem \ref{th:faces-of-BTstar2k} allows to recover the
description of the normal cones of $\LpSupportBall{\Ball}{p}{k}$ given in
\cite[Proposition~23]{LeFranc-Chancelier-DeLara:2022}. In our setting, the
normal cone of $\LpSupportBall{\Ball}{p}{k}$ at one of its exposed faces $F$
refers to the closure of the set of the vectors $\dual$ in $\mathbb{R}^{\spacedim}$ such that
$F$ is the face of $\LpSupportBall{\Ball}{p}{k}$ exposed by $\dual$. The normal fans
of $\LpSupportBall{\Ball}{2}{2}$ and $\LpTopBall{\Ball}{2}{2}$ are illustrated
in Fig.~\ref{CDLDP.sec.3.3.fig.3} when $d$ is equal to~$3$. Fig.~\ref{CDLDP.sec.3.3.fig.3} only
shows a portion of these fans but both can be reconstructed by symmetry.
 
\begin{figure}
  \begin{center}
    \mbox{\includegraphics[scale=1.3]{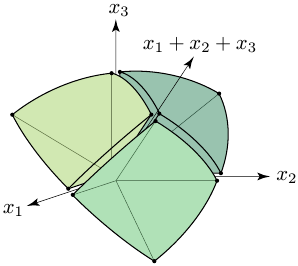}} \qquad
    \mbox{\includegraphics[scale=1.3]{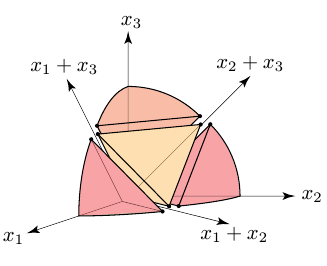}}    
  \end{center}
  \caption{Portions of the normal fan of $\LpSupportBall{\Ball}{2}{2}$ (left) and
     of $\LpTopBall{\Ball}{2}{2}$ (right) when $\spacedim=3$}
  \label{CDLDP.sec.3.3.fig.3}
\end{figure}

\begin{theorem}[{{\cite[Proposition~23]{LeFranc-Chancelier-DeLara:2022}}}]
  \label{CDLDP.sec.3.3.thm.2}
For any~$p$ satisfying $1<p<+\infty$, the normal cones of $\LpSupportBall{\Ball}{p}{k}$ at its exposed faces are the
  sets of the form
  \begin{equation}\label{CDLDP.sec.3.3.thm.2.eq.0}
    \overline{\mathrm{cone}}{ \bsetco{\dual\in\mathbb{R}^{\spacedim}}{\pi_{\overline{L}_k(z)}y=z,\,\overline{L}_k(\dual) =\overline{L}_k(z)}}
    \eqfinv 
  \end{equation}
  where $\overline{\mathrm{cone}}(X)$ denotes the closure of the cone spanned by $X$
  and~$z$ is a nonzero vector from $\mathbb{R}^{\spacedim}$ such that $z$ coincides with $\pi_{\overline{L}_k(z)}z$.
\end{theorem}
\begin{proof}
  Consider a nonzero vector~$z$ in $\mathbb{R}^{\spacedim}$ and denote
  \begin{equation*}
    G=F_\perp\Bigl(\LpSupportBall{\Ball}{p}{k},z\Bigr)\mbox{.}
  \end{equation*}
  According to Theorem \ref{th:faces-of-BTstar2k}, we get that 
  \begin{equation}\label{CDLDP.sec.3.3.thm.2.eq.0.5}
    G=\closedconvexhull\Bsetco{ v_p(\pi_{\IndexSubset\opt}z)}{\IndexSubset\opt\in\argmax_{\cardinal{\IndexSubset}\leq k} \norm{\projection_{\IndexSubset}z}_1}\mbox{.}
  \end{equation}
  It suffices to determine all the vectors $\dual$ such that $G$ is the face of
  $\LpSupportBall{\Ball}{p}{k}$ exposed by $\dual$. First observe that a subset
  $\IndexSubset\opt$ of $\llbracket1,{\spacedim}\rrbracket$ belongs to
  ${\argmax_{\cardinal{\IndexSubset}\leq k}{\|{\projection_{\IndexSubset}z}}\|_1}$ if
  and only if
  \begin{equation}\label{CDLDP.sec.3.3.thm.2.eq.1}
    L_k(z)\subset\IndexSubset\opt\subset\overline{L}_k(z)
  \end{equation}
  and either $m_k(z)$ is equal to $0$ or $\IndexSubset\opt$ has exactly $k$
  elements.
  In particular, by Theorem \ref{th:faces-of-BTstar2k}, we get that 
  \begin{equation}
    G=F_\perp\Bigl(\LpSupportBall{\Ball}{p}{k},\pi_{\overline{L}_k(z)}z\Bigr)
    \eqfinv
  \end{equation}
  and we can assume, without loss of generality, that $z$ coincides with
  $\pi_{\IndexSubset\opt}z$. We will treat two separate cases depending on whether
  $m_k(z)$ is equal to $0$ or not. First, if $m_k(z)$ is equal to $0$, then
  $\overline{L}_k(z)$ is equal to $\llbracket1,{\spacedim}\rrbracket$ and
  (\ref{CDLDP.sec.3.3.thm.2.eq.0}) just states that the normal cone of
  $\LpSupportBall{\Ball}{p}{k}$ at $G$ is the half-line spanned by~$z$. However,
  in that case, Theorem~\ref{th:faces-of-BTstar2k} states that $G$ is equal to
  $\{v_p(z)\}$. Hence if $\dual$ is another nonzero vector such that $G$ is the face
  of $\LpSupportBall{\Ball}{p}{k}$ exposed by $\dual$, the points $v_p(z)$ and
  $v_p(\dual)$ must coincide, which implies that $x$ and $\dual$ are multiples of one
  another by a positive factor and that the normal cone of
  $\LpSupportBall{\Ball}{p}{k}$ at $G$ is the half-line spanned by~$z$, as
  desired.

  Now assume that $m_k(z)$ is not equal to $0$ and consider a nonzero vector $\dual$
  such that $G$ is the face of $\LpSupportBall{\Ball}{p}{k}$ exposed by $\dual$. It
  follows from Theorem \ref{th:faces-of-BTstar2k} that
  \begin{equation}\label{CDLDP.sec.3.3.thm.2.eq.2}
    G=\closedconvexhull\Bsetco{ v_p(\pi_{\IndexSubset\opt}y)}{\IndexSubset\opt\in\argmax_{\cardinal{\IndexSubset}\leq k}
      \norm{\projection_{\IndexSubset}y}_1}
    \eqfinp
  \end{equation}
  Consider a subset $\IndexSubset\opt$ of $\llbracket1,{\spacedim}\rrbracket$ such that
  \[
  \IndexSubset\opt\in\argmax_{\cardinal{\IndexSubset}\leq k} \norm{\projection_{\IndexSubset}z}_1
  \eqfinp
  \]
  Since $m_k(z)$ is not equal to $0$, $\IndexSubset\opt$ contains exactly $k$
  elements and $|z_i|$ is nonzero when $i$ belongs to $\IndexSubset\opt$. By
  construction, a coordinate of $v_p(\pi_{\IndexSubset\opt}z)$ is nonzero if and
  only if the corresponding coordinate of $\pi_{\IndexSubset\opt}z$ is nonzero. As
  a consequence, according to (\ref{CDLDP.sec.3.3.thm.2.eq.0.5}),
  $v_p(\pi_{\IndexSubset\opt}z)$ is the unique vertex of $G$ contained in
  $\FlatRR_{\IndexSubset\opt}$. It then follows from
  (\ref{CDLDP.sec.3.3.thm.2.eq.2}) that
  \begin{equation}\label{CDLDP.sec.3.3.thm.2.eq.3}
    \argmax_{\cardinal{\IndexSubset}\leq k} \norm{\projection_{\IndexSubset}y}_1=\argmax_{\cardinal{\IndexSubset}\leq k}
    \norm{\projection_{\IndexSubset}z}_1
    \eqfinv
  \end{equation}
  and that, for every set $\IndexSubset\opt$ contained in
  $\argmax_{\cardinal{\IndexSubset}\leq k} \norm{\projection_{\IndexSubset}z}_1$,
  \begin{equation}\label{CDLDP.sec.3.3.thm.2.eq.4}
    v_p(\pi_{\IndexSubset\opt}y)=v_p(\pi_{\IndexSubset\opt}z)\eqfinp
  \end{equation}
  According to Proposition \ref{CDLDP.sec.3.3.prop.1} and to
  (\ref{CDLDP.sec.3.3.thm.2.eq.3}), $\overline{L}_k(z)$ and $\overline{L}_k(\dual)$
  coincide. Now recall that the normal cones of $B_p$ at its proper exposed~faces are
  half lines incident to the origin of $\mathbb{R}^{\spacedim}$. Hence, according to
  (\ref{CDLDP.sec.3.3.thm.2.eq.4}), there exists a positive number
  $\alpha_{\IndexSubset\opt}$ such that
  \begin{equation*}
    \pi_{\IndexSubset\opt}y=\alpha_{\IndexSubset\opt}\pi_{\IndexSubset\opt}z
    \eqfinp
  \end{equation*}
  It remains to show that the value of $\alpha_{\IndexSubset\opt}$ does not depend on
  $\IndexSubset\opt$. Indeed, this will imply that, up to a positive
  multiplicative factor, $\pi_{\overline{L}_k(z)}\dual$ coincides with~$z$ and results
  in the desired form (\ref{CDLDP.sec.3.3.thm.2.eq.0}) for the normal cone of $\LpSupportBall{\Ball}{p}{k}$ at
  $G$. If $L_k(z)$ is nonempty, this is immediate. Indeed,
  \begin{equation*}
    \alpha_{\IndexSubset\opt}=\frac{\dual_i}{z_i}
    \eqfinv
  \end{equation*}
  for any element $i$ of $L_k(z)$ and as a consequence,
  $\alpha_{\IndexSubset\opt}$ does not depend on $\IndexSubset\opt$. Therefore,
  assume that $L_k(z)$ is empty. In that case, $z_i$ is equal to $m_k(z)$ when
  $i$ belongs to $\overline{L}_k(z)$, and equal to $0$ otherwise. Moreover, the
  sets $\IndexSubset\opt$ are precisely the size $k$ subsets of
  $\overline{L}_k(z)$. If $k$ is equal to $1$, these sets are all the singletons
  from $\overline{L}_k(z)$ and this implies that $L_k(\dual)$ is also
  empty. Therefore $\dual_i$ is equal to $m_k(\dual)$ when $i$ belongs to
  $\overline{L}_k(\dual)$ which implies that $\alpha_{\IndexSubset\opt}$ does not depend
  on $\IndexSubset\opt$.

  Finally assume that $k$ is at least $2$ and observe that, for any two
  nondisjoint size $k$ subsets $\IndexSubset\opt$ and
  $\widetilde{\IndexSubset}\opt$ of $\overline{L}_k(z)$, the values of
  $\alpha_{\IndexSubset\opt}$ and $\alpha_{\widetilde{\IndexSubset}\opt}$ necessarily
  coincide. As $k$ is at least $2$, the graph whose vertices are the size $k$
  subsets of $\overline{L}_k(z)$ and whose edges connect two of them when they
  are nondisjoint is connected; it follows that $\alpha_{\IndexSubset\opt}$ does not
  depend on~$\IndexSubset\opt$.
\end{proof}

By analogy with the polytopal case, the \emph{normal fan} of
$\LpSupportBall{\Ball}{p}{k}$ refers to the set
$\NORMALCONE(\LpSupportBall{\Ball}{p}{k})$ of its normal cones. It is a
consequence of Theorem \ref{CDLDP.sec.3.3.thm.2} that the normal cones and,
therefore, the normal fan of $\LpSupportBall{\Ball}{p}{k}$ do not depend on~$p$
when $1<p<+\infty$. Interestingly, $\Ball_p$ has the same property: its normal cones
are $\{0\}$ and the half-lines incident to the origin independently on~$p$ when
$1<p<+\infty$. We show that $\NORMALCONE(\LpSupportBall{\Ball}{p}{k})$ refines
$\NORMALCONE(\LpSupportBall{\Ball}{\infty}{k})$ in the sense of \cite[Section
7]{Ziegler1995}.
\begin{corollary}\label{CDLDP.sec.3.2.cor.2}
For any~$p$ satisfying $1<p<+\infty$, every cone from $\NORMALCONE(\LpSupportBall{\Ball}{p}{k})$ is contained in a
  cone from $\NORMALCONE(\LpSupportBall{\Ball}{\infty}{k})$.
\end{corollary}

\begin{proof}
  Since the normal cones of $\LpSupportBall{\Ball}{p}{k}$ do not depend on $p$,
  it suffices to prove the statement when $p$ is equal to $2$. Consider a cone
  $C$ in $\NORMALCONE(\LpSupportBall{\Ball}{2}{k})$. If $C$ is empty or equal to
  $\{0\}$, then this is immediate. Assume that $C$ is the normal cone of
  $\LpSupportBall{\Ball}{2}{k}$ at an exposed face which we will denote by
  $\Face$. According to Theorem \ref{CDLDP.sec.3.3.thm.2}, there exists a
  nonzero vector~$z$ in $\mathbb{R}^{\spacedim}$ that coincides with
  $\pi_{\overline{L}_k(z)}z$ such that
  \begin{equation}\label{eq:ConeC}
    C= \closedcone{\bsetco{\dual\in\mathbb{R}^{\spacedim}}
      {\pi_{\overline{L}_k(z)}y=z,\overline{L}_k(\dual)=\overline{L}_k(z)}}
    \eqfinp
  \end{equation}
  Observe that if $m_k(z)$ is equal to $0$, then $\overline{L}_k(z)$ is equal to
  $\llbracket1,{\spacedim}\rrbracket$. In that case, $C$ is the half-line
  spanned by~$z$ and it is contained in a normal cone of
  $\LpSupportBall{\Ball}{\infty}{k}$. Therefore, we shall assume from now on that
  $m_k(z)$ is positive. Consider a size $k$ set $\IndexSubset\opt$ such that
  $L_k(z)\subset{\IndexSubset\opt}\subset\overline{L}_k(z)$. Since $m_k(z)$ is positive,
  $z_i$ is nonzero when $i$ belongs to $\IndexSubset\opt$. Denote by $\dual$ the
  vector such that
  \begin{equation}
    y_i=
    \begin{cases}
      -1 & \mbox{if }z_i<0\mbox{ and }i\in{\IndexSubset\opt} \eqfinv \\
      1 & \mbox{if }z_i>0\mbox{ and }i\in{\IndexSubset\opt} \eqfinv \\
      0 & \mbox{if }i\not\in{\IndexSubset\opt}\eqfinp
    \end{cases}
  \end{equation}
  By construction, $\dual$ has exactly $k$ nonzero coordinates. Denote
  \begin{equation}
    F=\convexhull\biggl(\ExposedFace(\beta_{\spacedim},\dual)\cup
    {\frac{1}{k}\ExposedFace(\gamma_{\spacedim},\dual)}\biggr)
    \eqfinp
  \end{equation}
  According to Theorem \ref{CDLDP.sec.3.2.thm.1}, $\Face$ is a facet of
  $\LpTopBall{\Ball}{1}{k}$. We will show that $C$ is contained in the cone
  spanned by $\Face$. By polarity, this cone belongs to
  $\NORMALCONE(\LpSupportBall{\Ball}{\infty}{k})$ and this will therefore prove the
  corollary. Consider a vector $\primal$ such that
  $\pi_{\overline{L}_k(z)}\primal$ is equal to~$z$ and $\overline{L}_k(x)$ to
  $\overline{L}_k(z)$. It suffices to show that $\primal$ belongs to the cone spanned
  by $\Face$. Indeed, since that cone is closed, it follows from
  (\ref{eq:ConeC}) that $C$ must be contained in it. Observe that $x$ can be
  decomposed as
  \begin{equation}
    x=\pi_{\IndexSubset\opt}\primal+\pi_{\llbracket1,{\spacedim}\rrbracket\mathord{\setminus}\IndexSubset\opt}\primal
    \eqfinp
  \end{equation}
  By construction, $x_i$ is equal to $m_k(z)y_i$ when $i$ belongs to
  $\IndexSubset\opt\mathord{\setminus}L_k(z)$. Hence,
  \begin{equation}
    \pi_{\IndexSubset\opt}\primal=\pi_{L_k(z)}\primal-m_k(z)\pi_{L_k(z)}y+m_k(z)y
    \eqfinp
  \end{equation}
  Since $\pi_{\overline{L}_k(z)}\primal$ is equal to~$z$ and $L_k(z)$ is a subset of
  $\overline{L}_k(z)$, the first term in the right-hand side is equal to
  $\pi_{L_k(z)}z$. As a consequence, $x$ can be rewritten into $x'+x''$ where
  \begin{equation}
    \left\{
    \begin{array}{l}
      \primal'=\pi_{L_k(z)}z-m_k(z)\pi_{L_k(z)}\dual  \eqfinv \\[\smallskipamount]
      \primal''=m_k(z)y+\pi_{\llbracket1,{\spacedim}\rrbracket\mathord{\setminus}\IndexSubset\opt}\primal\eqfinp
    \end{array}
  \right.
  \end{equation}
  Observe that $\primal'$ is contained in the cone spanned by
  $\ExposedFace(\beta_{\spacedim},\dual)$ and $\primal''$ in the cone spanned by
  $\ExposedFace(\gamma_{\spacedim},\dual)$. Hence, $x'+x''$ is contained in the
  cone spanned by $\Face$, as desired.
\end{proof}

We recall that the \emph{hypersimplex~$\delta_{d,k}$} is the convex hull of the
vertices of $[0,1]^{\spacedim}$ that have exactly $k$ nonzero coordinates. These
polytopes appear in algebraic combinatorics
\cite{GabrielovGelfandLosik1975,GelfandGoreskyMacPhersonSerganova1987} and in
convex geometry \cite{DezaHiriart-UrrutyPournin2021,Pournin2023,Pournin2024}. By
extension, we call \emph{hypersimplex} any polytope that coincides, up to a
bijective affine transformation, with the convex hull of such a subset of
vertices of the hypercube. It is observed in
\cite{DezaHiriart-UrrutyPournin2021} that $\LpSupportBall{\Ball}{\infty}{k}$ can be
decomposed into a union of hypersimplices with pairwise disjoint interiors and
that the proper faces of $\LpSupportBall{\Ball}{\infty}{k}$ are either hypersimplices
or isometric copies of $\LpSupportBall{\Ball}{\infty}{k}$, but for another ambient
dimension than~${\spacedim}$.
The situation for $\LpSupportBall{\Ball}{p}{k}$ is different as all of its
proper faces (including the nonexposed ones) are hypersimplices.

\begin{corollary}\label{CDLDP.sec.3.3.cor.1}
For any~$p$ satisfying $1<p<+\infty$,
all the proper faces of $\LpSupportBall{\Ball}{p}{k}$ are hypersimplices.
\end{corollary}
\begin{proof}
  Consider a proper face $\Face$ of $\LpSupportBall{\Ball}{p}{k}$. Let us first
  assume that $\Face$ is exposed by a vector $\dual$. If $m_k(\dual)$ is equal
  to $0$, then for every subset $\IndexSubset\opt$ of
  $\llbracket1,{\spacedim}\rrbracket$ that belongs to
  ${\argmax_{\cardinal{\IndexSubset}\leq k}{\|{\projection_{\IndexSubset}\dual}}\|_p}$,
  \begin{equation}
    \pi_{\IndexSubset\opt}\dual=\dual
  \end{equation}
  It then follows from Theorem \ref{th:faces-of-BTstar2k} that $\Face$ is the
  convex hull of a single point and therefore a $0$-dimensional
  hypersimplex. Let us now assume that $m_k(\dual)$ is positive.

  Denote $|\overline{L}_k(\dual)\mathord{\setminus}L_k(\dual)|$ by $n$ and let us
  identify $\mathbb{R}^n$ with the subspace of $\mathbb{R}^{\spacedim}$ made up of the points
  $\primal$ such that $\primal_i$ is equal to $0$ when $i$ belongs to either
  $L_k(\dual)$ or
  $\llbracket1,{\spacedim}\rrbracket\mathord{\setminus}\overline{L}_k(\dual)$. Denote
  by   $P$ the orthogonal projection of $\Face$ on $\mathbb{R}^n$. It follows from Theorem
  \ref{th:faces-of-BTstar2k} that $P/m_k(\dual)$
  is, up to flipping the signs of the coordinates, the convex hull of the
  points in $\{0,1\}^{\spacedim}$ with exactly $k-|L_k(\dual)|$ nonzero
  coordinates, as desired.

  Now assume that $\Face$ is not an exposed face of
  $\LpSupportBall{\Ball}{p}{k}$. In that case, $\Face$ is a proper face of an
  exposed face of $\LpSupportBall{\Ball}{p}{k}$. Hence by the above, $\Face$ is
  proper face of a hypersimplex. As all the proper faces of a hypersimplex are
  lower-dimensional hypersimplices, this completes the proof.
\end{proof}

\section{Conclusion}
Our original motivation, as exposed in~\S\ref{Motivation}, was to identify a class of norms which, when added as a
penalty term, promote sparsity in an optimization problem, \emph{but within a
  given sparsity budget}~$k$.
For this purpose, we have studied in Sect.~\ref{Support_identification} the
exposed faces of closed convex sets generated by $k$-sparse vectors, hence whose
extreme points are $k$-sparse.  We also have deduced support identification from
dual information.
Thus equipped, we have focused in
Sect.~\ref{The_case_of_generalized_top-k_and_k-support_norms} on the exposed faces of
the unit balls of so-called generalized $k$-support dual~norms, constructed from
$k$-sparse vectors and from a source norm.  In the cases of orthant-monotonic
and orthant-strictly monotonic source norms, we have obtained a characterization
of the intersection of the $k$-sparse vectors with the exposed faces of the $k$-support
norm.  Theorem~\ref{th:support_identification} makes the link with our original
motivation: we have provided dual conditions under which the primal optimal
solution of a minimization problem, penalized by a $k$-support dual~norm, is
$k$-sparse.
Going back to the original work of Tibshirani~\cite{Tibshirani:1996}, the
intuition --- behind proposing least-squares regression with an $\ell_1$-norm penalty
to achieve sparsity --- is that the kinks of the $\ell_1$-unit ball are located at
sparse points (see Figure~\ref{fig:Figure2_Tibshirani}).  In
Sect.~\ref{Geometry_of_lptopnorm{q}{k}_and__lpsupportnorm{p}{k}s_unit_balls}, we
have gone on in that direction by providing geometric descriptions of the faces
and normal cones of the unit balls of \lptopnorm{q}{k}s and
\lpsupportnorm{p}{k}s.  By contrast with the $\ell_1$-unit ball, faces intersect in
a subtle way, mixing kinks and smoother parts, as illustrated by
Figure~\ref{CDLDP.sec.3.3.fig.1} (left).
So, guided by sparsity, we have moved in this paper from optimization to the
geometry of unit balls.  \medskip

\noindent\textbf{Acknowledgements.} The third author is partially supported by
the Natural Sciences and Engineering Research Council of Canada Discovery Grant
program number RGPIN-2020-06846.

\appendix

\section{Proofs of Theorem~\ref{th:Support_identification} and Corollary~\ref{cor:Support_identification}} 
\label{Proofs}

We begin the section by proving the following preparatory lemma.

\begin{lemma}
  \label{lem:Support_identification}
  Consider a (primal) nonempty subset \( \Primal \) of \( \PRIMAL \).
  For any (dual) vector~$\dual$ contained in $\DUAL$ and any subset
  $\IndexSubset$ of $\Vset$, we have that
  \begin{subequations}
    \begin{align}
      \argmax_{\primal\in\ProjectionIndexSubset\np{\Primal}}\nscalv{\primal}{\dual}
      &=
    \ProjectionIndexSubset\Bp{\argmax_{\primalbis\in\Primal}\nscalv{\primalbis}{\ProjectionIndexSubset\dual}}
        \label{lem_eq:Support_identification_argmax}
\\        
    \max_{\primal\in\ProjectionIndexSubset\np{\Primal}}\nscalv{\primal}{\dual}
    &=
    \SupportFunction{\closedconvexhull\Primal}\np{\ProjectionIndexSubset\dual}
    \eqfinp 
    \label{lem_eq:Support_identification_max}
    \end{align}    
  \end{subequations}
\end{lemma}
\begin{proof}
 We prove~\eqref{lem_eq:Support_identification_argmax} by establishing 
  the following sequence of equivalences\footnote{%
    See Footnote~\ref{ft:optimal_solution} for the notation~$\opt$.}:
   \begin{align*}
    \primal\opt \in
    \argmax_{\primal\in\ProjectionIndexSubset\np{\Primal}}\nscalv{\primal}{\dual}
    \iff
    &
      \primal\opt \in \ProjectionIndexSubset\np{\Primal} \mtext{ and }
      \nscalv{\primal\opt}{\dual} \geq \nscalv{\primal}{\dual}
      \eqsepv \forall \primal\in\ProjectionIndexSubset\np{\Primal}
      \tag{by definition of~\( \argmax_{\primal\in\ProjectionIndexSubset\np{\Primal}}\nscalv{\primal}{\dual}\)}
    \\
    \iff
     &
       \primal\opt \in \ProjectionIndexSubset\np{\Primal} \mtext{ and }
       \nscalv{\primal\opt}{\dual} \geq \nscalv{\ProjectionIndexSubset\primalbis}{\dual}
       \eqsepv \forall \primalbis\in\Primal
       \tag{by definition of~\( \ProjectionIndexSubset\np{\Primal} \)}
    \\
    \iff
     &
       \primal\opt \in \ProjectionIndexSubset\np{\Primal} \mtext{ and }
      \nscalv{\ProjectionIndexSubset\primal\opt}{\dual} \geq \nscalv{\ProjectionIndexSubset\primalbis}{\dual}
      \eqsepv \forall \primalbis\in\Primal
      \intertext{because \(   \primal\opt \in \ProjectionIndexSubset\np{\Primal} \) belongs to
      the image of the orthogonal projection~\( \ProjectionIndexSubset \),
      hence \( \primal\opt=  \ProjectionIndexSubset\primal\opt \),}       
       \iff
     &
       \primal\opt \in \ProjectionIndexSubset\np{\Primal} \mtext{ and }
       \nscalv{\primal\opt}{\ProjectionIndexSubset\dual} \geq \nscalv{\primalbis}{\ProjectionIndexSubset\dual}
       \eqsepv \forall \primalbis\in\Primal
       \tag{as the orthogonal projection~\( \ProjectionIndexSubset \) is
       self-adjoint, see~\eqref{eq:orthogonal_projection_self-dual}}
    \\
    \iff
    &
      \primal\opt \in \ProjectionIndexSubset\np{\Primal} \mtext{ and }
      \primal\opt\in \argmax_{\primalbis\in\Primal}\nscalv{\primalbis}{\ProjectionIndexSubset\dual}
       \tag{by definition of~\(\argmax_{\primalbis\in\Primal}\nscalv{\primalbis}{\ProjectionIndexSubset\dual}\)}
     \\
    \iff
    &
      \primal\opt \in \ProjectionIndexSubset\np{\Primal} \mtext{ and }
      \ProjectionIndexSubset\primal\opt\in\ProjectionIndexSubset
      \Bp{ \argmax_{\primalbis\in\Primal}\nscalv{\primalbis}{\ProjectionIndexSubset\dual}}
    \intertext{where $\implies$ follows by applying the
      mapping~$\ProjectionIndexSubset$ to \( \primal\opt\in
      \argmax_{\primalbis\in\Primal}\nscalv{\primalbis}{\ProjectionIndexSubset\dual}\);
      regarding $\impliedby$, let $\primal'\in
      \argmax_{\primalbis\in\Primal}\nscalv{\primalbis}{\ProjectionIndexSubset\dual}$
      be such that
      $\ProjectionIndexSubset\primal\opt=\ProjectionIndexSubset \primal'$; 
      as $\nscalv{\primalbis}{\ProjectionIndexSubset\dual}=
      \nscalv{\ProjectionIndexSubset\primalbis}{\ProjectionIndexSubset\dual}$,
      we deduce that $\ProjectionIndexSubset\primal'\in
      \argmax_{\primalbis\in\Primal}\nscalv{\primalbis}{\ProjectionIndexSubset\dual}$,
      hence that $\ProjectionIndexSubset\primal\opt=\ProjectionIndexSubset \primal'\in
      \argmax_{\primalbis\in\Primal}\nscalv{\primalbis}{\ProjectionIndexSubset\dual}$;
      as \( \primal\opt=  \ProjectionIndexSubset\primal\opt \), we conclude that
      \( \primal\opt\in
      \argmax_{\primalbis\in\Primal}\nscalv{\primalbis}{\ProjectionIndexSubset\dual} \),}
    \iff
    &
      \primal\opt \in \ProjectionIndexSubset\np{\Primal} \mtext{ and }
      \primal\opt\in\ProjectionIndexSubset\Bp{
      \argmax_{\primalbis\in\Primal}\nscalv{\primalbis}{\ProjectionIndexSubset\dual}}
      \tag{as \( \primal\opt=  \ProjectionIndexSubset\primal\opt \).}            
  \end{align*}
  \bigskip
  Finally~\eqref{lem_eq:Support_identification_max} holds true because,
  as the orthogonal projection~\( \ProjectionIndexSubset \) is self-adjoint
  (see~\eqref{eq:orthogonal_projection_self-dual}),
  \[ 
    \max_{\primal\in\ProjectionIndexSubset\np{\Primal}}\nscalv{\primal}{\dual}
    =
    \max_{\primalbis\in\Primal}\nscalv{\ProjectionIndexSubset\primalbis}{\dual}
    =     \max_{\primalbis\in\Primal}\nscalv{\primalbis}{\ProjectionIndexSubset\dual}
    =\SupportFunction{\closedconvexhull\Primal}\np{\ProjectionIndexSubset\dual}
    \eqfinv
  \]
  by definition~\eqref{eq:support_function} and the well-known property
  \( \SupportFunction{\Primal}=\SupportFunction{\closedconvexhull\Primal} \) 
  of the support function~\( \SupportFunction{\Primal} \).
\end{proof}

\subsection{Proof of Theorem~\ref{th:Support_identification}}

\begin{proof}
  Consider a number \( k \) in \( \Vset \), a (primal) nonempty subset \( \Primal \) of \( \PRIMAL\),
  and a (dual) vector \( \dual \) in \( \DUAL \). First observe that
  \begin{equation*}
  \emptyset \neq  \OptimalSupports{\Primal,k}{\dual}
  \subset \ba{ \IndexSubset \subset\Vset, \cardinal{\IndexSubset}\leq k} 
  \subset 2^{\Vset}.
  \end{equation*}

  Second, the last inclusion in Equation~\eqref{eq:sparse_atomic_set} follows from
  the definition~\eqref{eq:sparse_atomic_set} of~$\AtomicSet$
  as, by~\eqref{eq:LevelSet_lzero_FlatRR_IndexSubset},
  \begin{equation*}
  \AtomicSet
  = \bigcup_{\cardinal{\IndexSubset}\leq k}
  \ProjectionIndexSubset\np{\Primal}
  \subset \bigcup_{\cardinal{\IndexSubset}\leq k} \FlatRR_{\IndexSubset}
  = {\LevelSet{\lzero}{k}}
  \end{equation*}

  Third, we prove~\eqref{eq:Support_identification}.
  We have that
  \begin{align*}
    \primal\opt \in {\Atom_{k}}
    \cap
    \ExposedFace(\closedconvexhull{\Atom_{k}},\dual)
    &\iff
      \primal\opt \in \argmax_{\primal\in\Atom_{k}} \nscalv{\primal}{\dual}
      \intertext{as \( \ExposedFace(\closedconvexhull{\Atom_{k}},\dual)=
      \argmax_{\primal\in\closedconvexhull{\Atom_{k}}}\nscalv{\primal}{\dual}\) and
      \(  \max_{\primal\in\Atom_{k}}      \proscal{\primal}{\dual}=
      \max_{\primal\in\closedconvexhull{\Atom_{k}}}\proscal{\primal}{\dual} \),
      }
      &\iff 
        \primal\opt \in {\Atom_{k}} \text{ and }
        \nscalv{\primal\opt}{\dual} =
        \max_{\primal\in\bigcup_{\cardinal{\IndexSubset}\leq k}
        \ProjectionIndexSubset\np{\Primal} }\nscalv{\primal}{\dual}
        \intertext{using the definition~\eqref{eq:sparse_atomic_set} of~$\Atom_{k}$
        and where, in all this proof, the subscript \( \cardinal{\IndexSubset}\leq k \)
        has to be understood as \( \IndexSubset \subset\Vset, \cardinal{\IndexSubset}\leq k \), }
        & \iff 
          \primal\opt \in {\Atom_{k}} \text{ and }
          \nscalv{\primal\opt}{\dual} = \max_{\cardinal{\IndexSubset}\leq k}
          \max_{\primal\in\ProjectionIndexSubset\np{\Primal} }\nscalv{\primal}{\dual}
    \\
    & \iff
      \primal\opt \in {\Atom_{k}} \text{ and there exists }
      \IndexSubset\opt \subset \Vset \eqsepv \cardinal{\IndexSubset\opt}\leq k
      \text{ such that}  
    \\
    &\hphantom{\iff\text{ }}
      \primal\opt \in \pi_{\IndexSubset\opt}\np{\Primal} \text{ and }
      \nscalv{\primal\opt}{\dual} = \max_{\cardinal{\IndexSubset}\leq k}
      \max_{\primal\in\ProjectionIndexSubset\np{\Primal} }\nscalv{\primal}{\dual}
      \intertext{because, by definition~\eqref{eq:sparse_atomic_set}
      of~$\Atom_{k}$, \( \primal\opt \in \bigcup_{\cardinal{\IndexSubset}\leq k}
      \ProjectionIndexSubset\np{\Primal} \),
      hence there exists \( \IndexSubset\opt \subset \Vset \) with \(
      \cardinal{\IndexSubset\opt}\leq k\)
      such that \( \primal\opt \in \pi_{\IndexSubset\opt}\np{\Primal} \), }
    & \iff
      \primal\opt \in {\Atom_{k}} \text{ and there exists }
      \IndexSubset\opt \subset \Vset \eqsepv \cardinal{\IndexSubset\opt}\leq k
      \text{ such that}  
    \\
    &\hphantom{\iff\text{ }}
      \primal\opt \in \pi_{\IndexSubset\opt}\np{\Primal} \text{ and }
      \nscalv{\primal\opt}{\dual} = \max_{\cardinal{\IndexSubset}\leq k}
      \max_{\primal\in\ProjectionIndexSubset\np{\Primal} }\nscalv{\primal}{\dual}
    \\
    &\hphantom{\iff\text{ }}\text{ and }
      \IndexSubset\opt \in
      \argmax_{\cardinal{\IndexSubset}\leq k}
      \max_{\primal\in\ProjectionIndexSubset\np{\Primal}}\nscalv{\primal}{\dual}
      \tag{by definition of \( \argmax_{\cardinal{\IndexSubset}\leq k}
      \max_{\primal\in\ProjectionIndexSubset\np{\Primal}}\nscalv{\primal}{\dual} \)}
    \\
    & \iff
      \primal\opt \in {\Atom_{k}} \text{ and there exists }
      \IndexSubset\opt \in
      \argmax_{\cardinal{\IndexSubset}\leq k}
      \max_{\primal\in\ProjectionIndexSubset\np{\Primal}}\nscalv{\primal}{\dual}
      \text{ such that}  
    \\
    &\hphantom{\iff\text{ }}
      \primal\opt \in \pi_{\IndexSubset\opt}\np{\Primal} \text{ and }
      \nscalv{\primal\opt}{\dual} = \max_{\primal\in\pi_{\IndexSubset\opt}\np{\Primal} }\nscalv{\primal}{\dual}
      \intertext{because \( \IndexSubset\opt \in\argmax_{\cardinal{\IndexSubset}\leq k}
      \max_{\primal\in\ProjectionIndexSubset\np{\Primal}}\nscalv{\primal}{\dual}
      \)
      and by definition of \( \argmax_{\cardinal{\IndexSubset}\leq k}
      \max_{\primal\in\ProjectionIndexSubset\np{\Primal}}\nscalv{\primal}{\dual} \),
      }
      & \iff
      \text{there exists }
      \IndexSubset\opt \in
      \argmax_{\cardinal{\IndexSubset}\leq k}
      \max_{\primal\in\ProjectionIndexSubset\np{\Primal}}\nscalv{\primal}{\dual}
    \\
    &\hphantom{\iff\text{ }}\text{such that }
      \primal\opt \in \argmax_{\primal\in\pi_{\IndexSubset\opt}\np{\Primal}}\nscalv{\primal}{\dual}
      \intertext{because \( \primal\opt \in 
      \argmax_{\primal\in\pi_{\IndexSubset\opt}\np{\Primal}}\nscalv{\primal}{\dual}
      \)
      implies that \( \primal\opt \in \pi_{\IndexSubset\opt}\np{\Primal} \),
      hence that \( \primal\opt \in {\Atom_{k}} \), by Definition~\eqref{eq:sparse_atomic_set}
      of~$\Atom_{k}$,}
      &\iff\text{there exists }
        \IndexSubset\opt \in
        \argmax_{\cardinal{\IndexSubset}\leq k}
        \max_{\primal\in\ProjectionIndexSubset\np{\Primal}}\nscalv{\primal}{\dual}
    \\
    &\hphantom{\iff\text{ }}\text{such that }
      \primal\opt \in
      \pi_{\IndexSubset\opt}\Bp{\argmax_{\primalbis\in\Primal}\nscalv{\primalbis}{\pi_{\IndexSubset\opt}\dual}}
      \tag{by~\eqref{lem_eq:Support_identification_argmax}}
      \eqfinv
      \\
      &\iff\text{there exists }
        \IndexSubset\opt \in
        \argmax_{\cardinal{\IndexSubset}\leq k}
        \SupportFunction{\closedconvexhull\Primal}\np{\ProjectionIndexSubset\dual}
        \tag{as \( \max_{\primal\in\ProjectionIndexSubset\np{\Primal}}\nscalv{\primal}{\dual} =
        \SupportFunction{\closedconvexhull\Primal}\np{\ProjectionIndexSubset\dual} \)
        by~\eqref{lem_eq:Support_identification_max}}
    \\
    &\hphantom{\iff\text{ }}\text{such that }
      \primal\opt \in
      \pi_{\IndexSubset\opt}\Bp{\Primal \cap
    \ExposedFace(\closedconvexhull\Primal,\pi_{\IndexSubset\opt}\dual)}
      \intertext{as \( \argmax_{\primalbis\in\Primal}\nscalv{\primalbis}{\pi_{\IndexSubset\opt}\dual}=
    \Primal \cap
    \ExposedFace(\closedconvexhull\Primal,\pi_{\IndexSubset\opt}\dual) \) by
      definition~\eqref{eq:ExposedFace} of the exposed face \(
  \ExposedFace(\closedconvexhull\Primal,\pi_{\IndexSubset\opt}\dual) \),}
    &\iff
      \primal\opt \in \Bsetco{\ProjectionIndexSubsetOpt\bp{\Primal \cap
      \ExposedFace(\closedconvexhull\Primal,\ProjectionIndexSubsetOpt\dual)}}
      {\IndexSubset\opt \in \OptimalSupports{\Primal,k}{\dual}}
      \tag{by using the notation~\eqref{eq:OptimalSupports}
      \( \OptimalSupports{\Primal,k}{\dual}=
      \argmax_{\cardinal{\IndexSubset}\leq k}
      \SupportFunction{\closedconvexhull\Primal}\np{\ProjectionIndexSubset\dual}
      \) }
      \eqfinp
  \end{align*}
  Thus, we have proven the equality~\eqref{eq:Support_identification}.\\\\
  We now prove Equation~\eqref{eq:Face_description} which can be rewritten as 
  \[
    \ExposedFace\bp{\closedconvexhull\AtomicSet,\dual}=
    \closedconvexhull\widehat{F}
    \mtext{ where } 
    \widehat{F} = 
    \bsetco{\pi_{\IndexSubset\opt}\bp{\Primal \cap
        \ExposedFace(\closedconvexhull\Primal,\pi_{\IndexSubset\opt}\dual)}}
    {\IndexSubset\opt \in \OptimalSupports{\Primal,k}{\dual} }
  \]
  is the right-hand side term in equality~\eqref{eq:Support_identification},
  which now writes
  \( \AtomicSet \cap \ExposedFace\bp{\closedconvexhull\AtomicSet,\dual} =
  \widehat{F} \).
  
  First, we have that
  \( \ExposedFace\bp{\closedconvexhull\AtomicSet,\dual} \supset \AtomicSet \cap
  \ExposedFace\bp{\closedconvexhull\AtomicSet,\dual} = \widehat{F} \), where the
  last equality is Equation~\eqref{eq:Support_identification}.
  An exposed face is closed convex. Thus,
  $\closedconvexhull \widehat{F} \subset
  \ExposedFace(\closedconvexhull\Atom_{k},\dual)$.  Similarly, since
  $\ExposedFace(\closedconvexhull\Atom_{k},\dual)$ is exposed, an extreme
  point~$e$ of $\ExposedFace(\closedconvexhull\Atom_{k},\dual)$ is also an
  extreme point of $\closedconvexhull\Atom_{k}$. Thus, it is also contained
  in~$\Atom_{k}$.  Using
  \( \AtomicSet \cap \ExposedFace\bp{\closedconvexhull\AtomicSet,\dual} =
  \widehat{F} \) (new form of Equation~\eqref{eq:Support_identification}), we
  get that $e \in \widehat{F}$. Thus, as $\widehat{F}$ contains all the extreme
  points of $\ExposedFace(\closedconvexhull\Atom_{k},\dual)$ and a convex set is
  the convex hull of its extreme points, we obtain that
  $\ExposedFace(\closedconvexhull\Atom_{k},\dual) \subset \closedconvexhull
  \widehat{F}$.  Third, from
  $\closedconvexhull \widehat{F} \subset
  \ExposedFace(\closedconvexhull\Atom_{k},\dual)$ and
  $\ExposedFace(\closedconvexhull\Atom_{k},\dual) \subset \closedconvexhull
  \widehat{F}$, we conclude that
  $\ExposedFace(\closedconvexhull\Atom_{k},\dual) = \closedconvexhull
  \widehat{F}$, finally giving Equation~\eqref{eq:Face_description}.
\end{proof}

\subsection{Proof of Corollary~\ref{cor:Support_identification}}

\begin{proof}
  The implication~\eqref{eq:Support_identification_Support_atom_opt} is a
  consequence of~\eqref{eq:Support_identification}.
  Indeed, as
  \begin{equation*}
    \pi_{\IndexSubset\opt}\bp{\Primal \cap
      \ExposedFace(\closedconvexhull\Primal,\pi_{\IndexSubset\opt}\dual)} \subset
    \FlatRR_{\IndexSubset\opt},
  \end{equation*}
  the support of any point 
  \( \atomopt \) in \( \Atom_{k} \cap
  \ExposedFace(\closedconvexhull\Atom_{k},\dual) \) is
  included in one of the subsets
  \( \IndexSubset\opt \) contained in \( \OptimalSupports{\Primal,k}{\dual} \)
  by~\eqref{eq:Support_identification}.\\\\
  Implication~\eqref{eq:Support_identification_Support_atom_opt} can be
  interpreted as follows: since \( \atomopt \in \Atom_{k} \), it follows
  from~\eqref{eq:sparse_atomic_set} that
  \begin{equation*} 
    \atomopt \in \bigcup_{\cardinal{\IndexSubset}\leq k} \FlatRR_{\IndexSubset}
  \end{equation*}
  and since \( \atomopt \) belongs to
  \( \ExposedFace(\closedconvexhull\Atom_{k},\dual) \), we can be more precise
  and obtain from~\eqref{eq:Support_identification} that
  \begin{equation*}
    \atomopt \in \pi_{\IndexSubset\opt}\bp{\Primal \cap
      \ExposedFace(\closedconvexhull\Primal,\pi_{\IndexSubset\opt}\dual)}.
  \end{equation*}
  As a consequence, $\atomopt$ belongs to $\FlatRR_{\IndexSubset\opt}$ or, equivalently,
  \( \Support{\atomopt} \) is a subset of \( \IndexSubset\opt \); the possible supports
  of~$\atomopt$ are the
  $\IndexSubset\opt\in \OptimalSupports{\Primal,k}{\dual} $, determined by the
  dual vector~\( \dual \) by means of~\eqref{eq:OptimalSupports}.\\\\
  Implication~\eqref{eq:Support_identification_Support_primal_opt} is a
  direct consequence
  of~\eqref{eq:Face_description}. 
  Indeed, as any
  \( \primal \) in \( \ExposedFace(\closedconvexhull\Atom_{k},\dual) \) can be expressed as a
  convex combination of elements of~$\FlatRR_{\IndexSubset\opt}$, with
  $\IndexSubset\opt$ in $\OptimalSupports{\Primal,k}{\dual}$,
  the support of~\( \primal \) is necessarily a subset of
  \begin{equation*}
    \bigcup_{ \IndexSubset\opt \in \OptimalSupports{\Primal,k}{\dual}}
    \IndexSubset\opt
  \end{equation*}
  as desired.
\end{proof}

\section{Summary table of generalized top-$k$ and $k$-support norms and dual~norms}
\label{Summary_table_on_generalized_top-k_and_k-support_norms}

In Table~\ref{tab:generalized_support_norms}, we provide a summary of
generalized top-$k$ and $k$-support norms,
and top-$k$ and $k$-support dual~norms,
that appeared in \cite{Chancelier-DeLara:2022_OSM_JCA,
   Chancelier-DeLara:2022_SVVA}; they rely on restriction norms (see
 Table~\ref{tab:Restriction_norms}).
 
\begin{table}[hbtp]
  \centering
  \begin{tabular}{c|c|c|c}  
    \hline\hline 
    norm 
    & unit ball 
      & polar unit ball 
    & dual norm 
    \\
    (on \( \RR^{\spacedim} \))
    & (\( \subset \RR^{\spacedim} \))
      & (\( \subset \RR^{\spacedim} \))
    & (on \( \RR^{\spacedim} \))
    \\
    \hline
  \( \TripleNorm{\cdot} \)
    & \( \TripleNormBall \) 
      & \( \TripleNormDualBall = \PolarSet{\TripleNormBall} \)
    & \( \TripleNormDual{\cdot} \)
  \\
    \hline\hline\hline 
     norm 
    & unit ball
      & polar unit ball
    & dual norm
    \\
   (on~$\FlatRR_{K}$)
    &  (\( \subset \FlatRR_{K} \))
    &  (\( \subset \FlatRR_{K} \))
    &    (on~$\FlatRR_{K}$)
    \\
    \hline\hline
restriction of \( \TripleNorm{\cdot} \) to $\FlatRR_{K}$ &&&
    \\
    \( \TripleNorm{\cdot}_{K} \)
    & \( \FlatRR_{K} \cap \TripleNormBall \) 
      & \( \projection_{K}\np{ \TripleNormDualBall } \)
    & \( \TripleNorm{\cdot}_{K,\star} \)
    \\
    & straightforward
      & deduced from 
    &
    \\
    && \cite[Eq.~(13a)]{Chancelier-DeLara:2022_OSM_JCA} &
        \\
    \hline
dual norm of \( \TripleNorm{\cdot}_{K} \) (on $\FlatRR_{K}$) &&&
    \\
    \( \TripleNorm{\cdot}_{K,\star} \)
     & \( \projection_{K}\np{ \TripleNormDualBall } \)
      & \( \FlatRR_{K} \cap \TripleNormBall \)
    & \( \TripleNorm{\cdot}_{K} \)
    \\
    & by polarity & by polarity & 
        \\
    & 
      &   \cite[Eq.~(13)]{Chancelier-DeLara:2022_CAPRA_OPTIMIZATION}        &
    \\
    & & \cite[Eq.~(13b)]{Chancelier-DeLara:2022_OSM_JCA}         &
    \\
    \hline\hline
restriction of \( \TripleNormDual{\cdot} \) to $\FlatRR_{K}$ &&&
    \\
    \( \TripleNorm{\cdot}_{\star,K} \)
     & \( \FlatRR_{K} \cap \TripleNormDualBall \) 
      &  \( \projection_K\np{\TripleNormBall} \) 
    & \( \TripleNorm{\cdot}_{\star,K,\star} \)
    \\
    &
      & \cite[Eq.~(13a)]{Chancelier-DeLara:2022_OSM_JCA}
    &
      \\
    \hline
 dual norm of \( \TripleNorm{\cdot}_{\star,K} \) (on $\FlatRR_{K}$) &&&
    \\   \( \TripleNorm{\cdot}_{\star,K,\star} \)
       & \( \projection_K\np{\TripleNormBall} \) 
     &  \( \FlatRR_{K} \cap \TripleNormDualBall \) 
   & \( \TripleNorm{\cdot}_{\star,K} \)
    \\
    \hline\hline 
  \end{tabular}
  \caption{Restriction norms on \( \FlatRR_{K} \) in~\eqref{eq:FlatRR}, where $K
    \subset \ic{1,\spacedim}$ (see
    \cite[\S2.2.1]{Chancelier-DeLara:2022_SVVA},
    \cite[Definition~3.1]{Chancelier-DeLara:2022_CAPRA_OPTIMIZATION},
    \cite[Definition~1)]{Chancelier-DeLara:2022_OSM_JCA})} 
  \label{tab:Restriction_norms}
\end{table}

\begin{table}[hbtp]
  \centering
  \small
  \begin{tabular}{c|c|c|c}  
    \hline\hline 
    source norm 
    & unit ball 
    & polar unit ball 
    & dual norm 
    \\
    (on \( \RR^{\spacedim} \))
    & (\( \subset \RR^{\spacedim} \))
    & (\( \subset \RR^{\spacedim} \))
    & (on \( \RR^{\spacedim} \))
    \\
    \hline
    \( \TripleNorm{\cdot} \)
    & \( \TripleNormBall \) 
    & \( \TripleNormDualBall = \PolarSet{\TripleNormBall} \)
    & \( \TripleNormDual{\cdot} \)
    \\
    \hline\hline\hline
    norm 
    & unit ball
    & polar unit ball
    & dual norm
    \\
    (on \( \RR^{\spacedim} \))
    & (\( \subset \RR^{\spacedim} \))
    & (\( \subset \RR^{\spacedim} \))
    & (on \( \RR^{\spacedim} \))
    \\
    \hline\hline
    top-$k$ norm & & & $k$-support
    \\
    \( \TopNorm{\TripleNorm{\cdot}}{k} \)
    & \( \intercardK \Converse{\projection_{K}} \np{\FlatRR_{K} \cap \TripleNormBall} \)
    &  \( \closedconvexhull\bp{ \unioncardK 
      \projection_{K}\np{ \TripleNormDualBall } } \)
    & norm 
    \\
    \(= \supcardK \TripleNorm{\projection_K\cdot}_K \) 
    & straightforward & by polarity
    &
      \( \SupportNorm{\TripleNorm{\cdot}}{k} \) 
    \\
    \cite[Eq.~(14)]{Chancelier-DeLara:2022_OSM_JCA}
    & 
      \( =\intercardK \PolarSet{\bp{\projection_K\np{\TripleNormDualBall}}} \) 
    &
    &
    \\
    &  easily proved
    &
    &
    \\
    \hline
    $k$-support norm &&& top-$k$ norm
    \\
    \( \SupportNorm{\TripleNorm{\cdot}}{k} \)
    & \( \closedconvexhull\bp{\unioncardK
      \projection_{K}\np{ \TripleNormDualBall }  } \)
    &
      \(\intercardK \Converse{\projection_{K}} \np{\FlatRR_{K} \cap \TripleNormBall} \)
    & \( \TopNorm{\TripleNorm{\cdot}}{k} \) 
    \\
    \cite[Eq.~(15)]{Chancelier-DeLara:2022_OSM_JCA}
    & \cite[Eq.~(22)]{Chancelier-DeLara:2022_OSM_JCA}                  
    &     \( = \intercardK \PolarSet{\bp{\projection_K\np{\TripleNormDualBall}}} \) 
    &
    \\
    & & by polarity &
    \\
    \hline\hline
    top-$k$ dual~norm & & & $k$-support 
    \\
    \( \TopDualNorm{\TripleNorm{\cdot}}{k}
    \)
    & \(\intercardK \Converse{\projection_{K}} \np{\FlatRR_{K} \cap \TripleNormDualBall} \)
    &  \( \closedconvexhull\bp{\unioncardK 
      \projection_{K} \np{ \TripleNormBall } } \)
    &  dual~norm 
    \\
    \( =
    \supcardK \TripleNorm{\projection_K\cdot}_{\star,K} \)
    &      \( = \intercardK \PolarSet{\bp{\projection_K\np{\TripleNormBall}}} \) 
    &
    &
      \( \SupportDualNorm{\TripleNorm{\cdot}}{k} \)
    \\
    \cite[Eq.~(10)]{Chancelier-DeLara:2022_SVVA}
    &
    &
    &
    \\
    \hline
    \hline
    $k$-support dual~norm &&& top-$k$
    \\
    \( \SupportDualNorm{\TripleNorm{\cdot}}{k} \)
    & \( \closedconvexhull\bp{\unioncardK 
      \projection_{K} \np{ \TripleNormBall } } \)
    & \(\intercardK  \Converse{\projection_{K}} \np{\FlatRR_{K} \cap
      \TripleNormDualBall} \)
    &    dual~norm
    \\
    \cite[Eq.~(11)]{Chancelier-DeLara:2022_SVVA}
    &
    & \( = \intercardK 
      \PolarSet{\bp{\projection_K\np{\TripleNormBall}}} \)
    &
      \( \TopDualNorm{\TripleNorm{\cdot}}{k} \)
    \\
    \hline\hline
  \end{tabular}
  \caption{Generalized top-$k$ norms, $k$-support norms  (see
    \cite[Definition~9]{Chancelier-DeLara:2022_OSM_JCA}),
    and top-$k$ dual~norms,
    $k$-support dual~norms (see
    \cite[Definition~3]{Chancelier-DeLara:2022_SVVA}),
    indexed by $k \in \ic{1,\spacedim}$.
    We immediately see that the generalized
    top-$k$ dual~norms and $k$-support dual~norms --- deduced from a source norm --- 
    are exactly the generalized top-$k$ norms and $k$-support norms  --- but deduced
    from the dual norm of the source norm.}
  \label{tab:generalized_support_norms}
\end{table} 

\newcommand{\noopsort}[1]{} \ifx\undefined\allcaps\def\allcaps#1{#1}\fi


\begin{thebibliography}{10}

\bibitem{Argyriou-Foygel-Srebro:2012}
A.~Argyriou, R.~Foygel, and N.~Srebro.
\newblock Sparse prediction with the $k$-support norm.
\newblock In {\em Proceedings of the 25th International Conference on Neural
  Information Processing Systems - Volume 1}, NIPS'12, pages 1457--1465, USA,
  2012. Curran Associates Inc.

\bibitem{Bach-Jenatton-Mairal-Obozinski:2012}
F.~Bach, R.~Jenatton, J.~Mairal, and G.~Obozinski.
\newblock Optimization with sparsity-inducing penalties.
\newblock {\em Foundations and Trends\textsuperscript{\textregistered} in
  Machine Learning}, 4(1):1--106, January 2012.

\bibitem{Bauer-Stoer-Witzgall:1961}
F.~L. Bauer, J.~Stoer, and C.~Witzgall.
\newblock Absolute and monotonic norms.
\newblock {\em Numer. Math.}, 3:257--264, 1961.

\bibitem{Bauschke-Combettes:2017}
H.~H. Bauschke and P.~L. Combettes.
\newblock {\em Convex analysis and monotone operator theory in {H}ilbert
  spaces}.
\newblock CMS Books in Mathematics/Ouvrages de Math\'ematiques de la SMC.
  Springer-Verlag, New York, second edition, 2017.

\bibitem{Boyer-Chambolle-DeCastro-Duval-deGournay-Weiss:2019}
C.~Boyer, A.~Chambolle, Y.~{De Castro}, V.~Duval, F.~{de Gournay}, and
  P.~Weiss.
\newblock On representer theorems and convex regularization.
\newblock {\em SIAM Journal on Optimization}, 29(2):1260--1281, 2019.

\bibitem{Candes:2008}
E.~J. Cand{\`e}s.
\newblock The restricted isometry property and its implications for compressed
  sensing.
\newblock {\em Comptes Rendus Math{\'e}matique}, 346(9):589--592, 2008.

\bibitem{Candes:ICM2014}
E.~J. Cand\`es.
\newblock Mathematics of sparsity (and a few other things).
\newblock In {\em Proceedings of the International Congress of Mathematicians},
  2014.

\bibitem{Chancelier-DeLara:2022_SVVA}
J.-P. Chancelier and M.~{De Lara}.
\newblock Capra-convexity, convex factorization and variational formulations
  for the $\ell_0$ pseudonorm.
\newblock {\em Set-Valued and Variational Analysis}, 30:597--619, 2022.

\bibitem{Chancelier-DeLara:2022_CAPRA_OPTIMIZATION}
J.-P. Chancelier and M.~{De Lara}.
\newblock Constant along primal rays conjugacies and the $\ell_0$ pseudonorm.
\newblock {\em Optimization}, 71(2):355--386, 2022.

\bibitem{Chancelier-DeLara:2022_OSM_JCA}
J.-P. Chancelier and M.~{De Lara}.
\newblock Orthant-strictly monotonic norms, generalized top-$k$ and $k$-support
  norms and the $\ell_0$ pseudonorm.
\newblock {\em Journal of Convex Analysis}, 30(3):743--769, 2023.

\bibitem{Chandrasekaran-Recht-Parrilo-Willsky:2012}
V.~Chandrasekaran, B.~Recht, P.~A. Parrilo, and A.~S. Willsky.
\newblock The convex geometry of linear inverse problems.
\newblock {\em Foundations of Computational Mathematics}, 12(6):805--849, 2012.

\bibitem{DezaHiriart-UrrutyPournin2021}
A.~Deza, J.-B. Hiriart-Urruty, and L.~Pournin.
\newblock Polytopal balls arising in optimization.
\newblock {\em Contributions to Discrete Mathematics}, 16(3):125--138, 2021.

\bibitem{Drusvyatskiy-Vavasis-Wolkowicz:2015}
D.~Drusvyatskiy, S.~A. Vavasis, and H.~Wolkowicz.
\newblock Extreme point inequalities and geometry of the rank sparsity ball.
\newblock {\em Mathematical Programming}, 152:521--544, 2015.

\bibitem{Fadili-Malick-Peyre:2018}
J.~Fadili, J.~Malick, and G.~Peyr\'{e}.
\newblock Sensitivity analysis for mirror-stratifiable convex functions.
\newblock {\em SIAM Journal on Optimization}, 28(4):2975--3000, 2018.

\bibitem{Fan-Jeong-Sun-Friedlander:2020}
Z.~Fan, H.~Jeong, Y.~Sun, and M.~P. Friedlander.
\newblock Atomic decomposition via polar alignment.
\newblock {\em Foundations and Trends\textsuperscript{\textregistered} in
  Optimization}, 3(4):280--366, 2020.

\bibitem{GabrielovGelfandLosik1975}
A.~M. Gabri{\'e}lov, I.~M. Gel'fand, and M.~V. Losik.
\newblock Combinatorial computation of characteristic classes.
\newblock {\em Functional Analysis and its Applications}, 9:103--115, 1975.

\bibitem{GelfandGoreskyMacPhersonSerganova1987}
I.~M. Gel'fand, M.~Goresky, R.~D. MacPherson, and V.~Serganova.
\newblock Combinatorial geometries, convex polyhedra and {Schubert} cells.
\newblock {\em Advances in Mathematics}, 63:301--316, 1987.

\bibitem{Gotoh-Takeda-Tono:2018}
J.-Y. Gotoh, A.~Takeda, and K.~Tono.
\newblock {DC} formulations and algorithms for sparse optimization problems.
\newblock {\em Mathematical Programming}, 169(1):141--176, 2018.

\bibitem{Gotoh-Uryasev:2016}
J.-Y. Gotoh and S.~Uryasev.
\newblock Two pairs of families of polyhedral norms versus $\ell _p$ -norms:
  proximity and applications in optimization.
\newblock {\em Mathematical Programming}, 156:391--431, 2016.

\bibitem{Gries:1967}
D.~Gries.
\newblock Characterization of certain classes of norms.
\newblock {\em Numerische Mathematik}, 10:30--41, 1967.

\bibitem{Gries-Stoer:1967}
D.~Gries and J.~Stoer.
\newblock Some results on fields of values of a matrix.
\newblock {\em SIAM Journal on Numerical Analysis}, 4(2):283--300, 1967.

\bibitem{Jenatton-Audibert-Bach:2011}
R.~Jenatton, J.-Y. Audibert, and F.~Bach.
\newblock Structured variable selection with sparsity-inducing norms.
\newblock {\em Journal of Machine Learning Research}, 12(84):2777--2824, 2011.

\bibitem{LeFranc-Chancelier-DeLara:2022}
A.~{Le Franc}, J.-P. Chancelier, and M.~{De Lara}.
\newblock The capra-subdifferential of the $\ell_0$ pseudonorm.
\newblock {\em Optimization}, 73(4):1229--1251, 2024.

\bibitem{McDonald-Pontil-Stamos:2016}
A.~M. McDonald, M.~Pontil, and D.~Stamos.
\newblock New perspectives on $k$-support and cluster norms.
\newblock {\em Journal of Machine Learning Research}, 17(155):1--38, 2016.

\bibitem{Mirsky:1960}
L.~Mirsky.
\newblock Symmetric gauge functions and unitarily invariant norms.
\newblock {\em The Quarterly Journal of Mathematics}, 11(1):50--59, 1960.

\bibitem{Obozinski-Bach:hal-01412385}
G.~Obozinski and F.~Bach.
\newblock {A unified perspective on convex structured sparsity: Hierarchical,
  symmetric, submodular norms and beyond}.
\newblock Preprint {\verb!hal-01412385!}, 2016.

\bibitem{Pournin2023}
L.~Pournin.
\newblock Shallow sections of the hypercube.
\newblock {\em Israel Journal of Mathematics}, 255:685--704, 2023.

\bibitem{Pournin2024}
L.~Pournin.
\newblock Local extrema for hypercube sections.
\newblock {\em Journal d'Analyse Math{\'e}matique}, 152:557--594, 2024.

\bibitem{Rockafellar-Wets:1998}
R.~T. Rockafellar and R.~J.-B. Wets.
\newblock {\em Variational Analysis}.
\newblock Springer-Verlag, Berlin, 1998.

\bibitem{Schneider:2014}
R.~Schneider.
\newblock {\em {Convex bodies: the {Brunn}-{Minkowski} theory}}.
\newblock Cambridge University Press, second edition, 2014.

\bibitem{Tibshirani:1996}
R.~Tibshirani.
\newblock Regression shrinkage and selection via the lasso.
\newblock {\em Journal of the Royal Statistical Society. Series B
  (Methodological)}, 58(1):267--288, 1996.

\bibitem{Wu-Ding-Sun-Toh:2014}
B.~Wu, C.~Ding, D.~Sun, and K.-C. Toh.
\newblock On the {Moreau-Yosida} regularization of the vector $k$-norm related
  functions.
\newblock {\em SIAM Journal on Optimization}, 24(2):766--794, 2014.

\bibitem{Ziegler1995}
G.~M. Ziegler.
\newblock {\em Lectures on Polytopes}, volume 152 of {\em Graduate Texts in
  Mathematics}.
\newblock Springer, 1995.

\end{thebibliography}
\end{document}